\documentclass[a4paper, 12pt]{article}
\usepackage[usenames]{color} 
\usepackage{amssymb,amsthm,amsmath,amsfonts} 
\usepackage{lipsum}
\usepackage[a4paper]{geometry}
\usepackage[latin1]{inputenc} 
\usepackage[francais, english]{babel}
\usepackage{empheq}
\usepackage[T1]{fontenc}
\usepackage{faktor}
\usepackage[all]{xy}
\usepackage{amsrefs}
\usepackage[hidelinks]{hyperref}

\theoremstyle{plain}
\newtheorem{theorem}{Theorem}[subsection]
\newtheorem{proposition}[theorem]{Proposition}

\newtheorem{lemma}[theorem]{Lemma}
\newtheorem{definition-proposition}[theorem]{Definition/Proposition}
\theoremstyle{definition}
\newtheorem{definition}[theorem]{Definition}
\theoremstyle{remark}
\newtheorem{remark}[theorem]{Remark}
\theoremstyle{example}
\newtheorem{example}[theorem]{Example}
\theoremstyle{notation}

\newcommand\C{\mathbb{C}}
\newcommand\R{\mathbb{R}}

\newcommand\h{{\mathcal{H}}}

\newcommand\p{{\text{\bfseries{P}}}}
\newcommand\Dp{\mathrm{d}_{\text{{\bfseries{P}}}}}
\newcommand\Dr{\mathrm{D}}
\DeclareMathOperator*{\esssup}{ess\,sup}
\DeclareMathOperator*{\adm}{adm}

\newcommand\dz{\frac{\partial}{\partial z}}

\newcommand\dt{\frac{\partial}{\partial t}}
\newcommand\dr{\text{dR}}
\newcommand\ru{\text{R}}

\newcommand{\Addresses}{{
  \bigskip
  \footnotesize

\textsc{Sorbonne Universit\'e, Universit\'e Paris Diderot, CNRS, Institut de Math\'ematiques de Jussieu-Paris Rive Gauche, IMJ-PRG, F-75005, Paris, France}\par\nopagebreak
  \textit{E-mail address} : \texttt{robin.timsit@imj-prg.fr}}}

\title{Quadratic differentials in spherical CR geometry}
\author{Robin Timsit}
\date{}

\begin{document}

\maketitle
\abstract
This paper deals with the notion of quadratic differential in spherical CR geometry (or more generally on strictly pseudoconvex CR manifolds). We get to this notion by studying a splitting of Rumin complex and discuss its first features such as trajectories and length. We also define several differential operators on quadratic differentials that lead to analogue of half-translation structures on spherical CR manifolds. Finally, we work on known examples of quasiconformal maps in the Heisenberg group with extremal properties and explicit how quadratic differentials are involved in those.

In addition, on our way to quadratic differentials, we define a differential complex on strictly pseudoconvex CR manifolds with a finite dimensional cohomology space. It leads to a new CR invariant that we compute for compact manifolds endowed with a CR action of the circle.

\section*{Introduction}
A $3$-dimensional strictly pseudoconvex CR manifold is a $3$-manifold endowed with a contact structure and a complex structure on it. This is a natural extension of the notion of Riemann surface. Holomorphic forms and holomorphic quadratic differentials on Riemann surfaces are classical and well-studied objects. One would like analogous objects on strictly pseudoconvex CR manifolds. On such manifolds, a Dolbeault type complex has been introduced by Kohn and Rossi \cite{KR} (see also \cite{Bog, FK}). It is known as the tangential Cauchy-Riemann or $\overline \partial_b$ complex and was used for various purposes. 

On the other hand, Rumin \cite {Rum} constructed an adapted differential complex on contact manifolds which works modulo the contact structure. In \cite{GL}, Garfield and Lee clarified a splitting of Rumin complex on strictly pseudoconvex CR manifolds that appeared implicitly in Rumin's work. This splitting of Rumin complex also led to a Dolbeault type complex whose cohomology is the same as the $\overline \partial_b$ one. Namely, the splitting of Rumin complex gives complex line bundles $\wedge^{p,q}$ (for $0\le p+q \le 3$, $0\le p \le 2$ and $0\le q \le 1$) whose sections are similar to $(p,q)$-forms and several linear differential operators on their sections. 

In order to define "holomorphic forms" on a CR manifold, it seems appropriate to consider already sections of $\wedge ^{1,0}$. On those sections, there are two differential operators: $\Dr ''$ with values in the space of sections of $\wedge^{1,1}$ and (contrary to the case of Riemann surfaces) $\Dr '$ with value in the space of sections of $\wedge^{2,0}$. Moreover, $(1,0)$-forms annihilating both $\Dr '$ and $\Dr ''$ are locally differentials of complex valued CR functions. Thus, we can consider the quotient of the space of $(1,0)$-forms annihilating both $\Dr '$ and $\Dr ''$ over the space of differentials of global CR functions. It turns out that this space can be naturally embedded in the first de Rham comohology space so that it is finite dimensional. It gives a CR invariant for compact strictly pseudoconvex CR manifolds that we compute for manifolds endowed with a CR action of the circle.

After that detour, we get to the heart of the present work which is quadratic differentials on spherical CR manifolds (or more generally on strictly pseudoconvex CR manifolds). Since we can understand sections of $\wedge^{1,0}$ as $(1,0)$-forms, it is natural to define quadratic differentials as sections of $\wedge^{1,0} \otimes \wedge^{1,0}$. The first geometric features such as trajectories and length follow easily from that definition. 

In order to define "holomorphy operators", we consider spherical CR manifolds. Those manifolds are CR manifolds locally equivalent to the unit sphere in $\C^2$ or equivalently, manifolds locally modeled on the Heisenberg group. 
The Heisenberg group $\h$ is $\C \times \R$ endowed with the group law
\[(z,t)(z',t')=(z+z', t+t'+2\Im (z\overline z')).\]
The CR structure on $\h$ is given by the CR distribution $V= span (Z)$ where $Z$ is the left invariant vector field
\[ Z= \dz +i\overline z \dt.\]
It is strictly pseudoconvex since $V\oplus \overline V = \ker (\omega)$ where $\omega = \mathrm{d}t -i\overline z \mathrm{d} z+i z \mathrm{d}\overline z$ is a contact form. Thus, a spherical CR manifold is a manifold $M$ endowed with an atlas of charts $(U_i , \varphi_i)$ with value in $\h$ and whose transition functions are CR diffeomorphisms of $\h$ (that is diffeomorphisms preserving $V$). On these manifolds, we manage to extend the operators $\Dr '$ and $\Dr ''$ into operators $\Dr_2 '$ and $\Dr_2 ''$ defined on quadratic differentials such that any quadratic differential annihilating both $\Dr_2 '$ and $\Dr_2 ''$ is locally the square of the differential of a CR function. Let us point out that one cannot extend directly $\Dr '$ and $\Dr''$ to quadratic differentials because the line bundle $\wedge^{1,0}$ isn't a CR line bundle in general (that is a line bundle whose transition functions are CR functions).

In order to have some structure on a spherical CR manifold similar to a half-translation structure, annihilating these two operators isn't enough. Namely, we need to get from quadratic differentials that are "locally the square of the differential of a CR function" to quadratic differentials that are "locally the square of the differential of the complex part of a chart of the manifold with value in the Heisenberg group". For that, we introduce a third differential operator $B_2$ on quadratic differentials. Then, a quadratic differential $q$ annihilating $\Dr_2 '$, $\Dr_2 ''$ and $B_2$ induce "natural coordinates" on the spherical CR manifold (that is coordinates $(z,t)$ where $q$ is $\mathrm{d}z^2$ up to a contact form). Moreover, two such natural coordinates $(z,t)$ and $(z',t')$ only differ by $(z',t') = \left(\pm z+z_0, t+t_0 + 2\Im(\pm z \overline z_0)\right)$.

As an end, we express the known examples of extremal quasiconformal map in the Heisenberg group as maps dilating trajectories of quadratic differentials. This is an interpretation of the results obtained in \cite{BFP, BFP2, Tim} in terms of trajectories of quadratic differentials. In particular, we find explicitly all quasiconformal maps between cylinders of spherical annuli in the Heisenberg group that dilate trajectories of a couple of given quadratic differentials. For cylinders, there is only a $1$-parameter family of such maps and for spherical annuli, a $2$-parameters family. This is a strong contrast with the classical case where quasiconformal maps dilating horizontal trajectories of quadratic differential generally exist at a large number.
\newline

The paper is organized as follow. The first section briefly presents Rumin complex and its Garfield-Lee decomposition. We also introduce another differential differential complex on strictly pseudoconvex CR manifold with a finite dimensional cohomology space that we compute in the case of manifolds endowed with a CR action of the circle. Section \ref{sec2} introduces quadratic differentials on (spherical or not) CR manifolds. For that, we begin with a description of the splitting of Rumin complex in spherical CR geometry. Then, we define quadratic differentials on strictly pseudoconvex CR manifolds and its first geometric features. After that, we define the operators $\Dr_2 '$, $\Dr_2 ''$ and $B_2$ on quadratic differentials and explain how one can extend them to operators on forms of any positive degree. The last section deals with extremal quasiconformal maps in the Heisenberg group as maps dilating trajectories of quadratic differentials. First, we recall the notions of quasiconformal maps and of modulus of a curve family in the case of the Heisenberg group. Then, we express the quadratic differentials involved in known examples of extremal quasiconformal maps in the Heisenberg group and maps dilating trajectories of those quadratic differentials.

\subsubsection*{Acknowledgement}
The author would like to thank Michel Rumin for fruitful discussions regarding this article and especially for the proofs of Propositions \ref{prop1.3.5} and \ref{prop1.3.6}.


\section{Rumin complex and its splitting}\label{sec1}
\subsection{Rumin complex on contact manifolds}
Let $M$ be an orientable $3$-manifold with a contact structure \p, that is, $\text{\p} \subset  TM$ is a rank $2$ distribution defined locally as the kernel of a $1$-form $\omega$ such that $\omega\wedge \mathrm{d} \omega$ is nowhere zero; and so, $(\mathrm{d}\omega)_{|\text{\p}}$ is nondegenerate. On such manifolds, Rumin \cite{Rum} defined a differential complex as follow.

Let $\Omega^\bullet(M)$ be the algebra of differential forms on $M$. If $U$ is a (sufficiently small) open subset of $M$, let $\omega_U$ be a $1$-form on $U$ with kernel $\p$. Then, two such forms only differ by multiplication with a nowhere vanishing real valued function so that the differential ideal generated by all these $\omega_U$, denoted $I^\bullet$, is well defined. To be precise, for every integer $k$, with the convention that $\Omega^k = \{0\}$ if $k<0$,
\begin{multline*}
I^k = \{\gamma \in \Omega^k (M) \ | \  \forall \omega_U \text{ with } \p = \ker \omega_U \text{ , }  \gamma = \omega_U\wedge\alpha + \mathrm{d}\omega_U \wedge \beta \\
\text{ with } \alpha \in \Omega^{k-1} (U), \beta \in \Omega^{k-2} (U) \}.
\end{multline*}
Consider also $F^\bullet$ the annihilator of $I^\bullet$. That is, for every integer $k$, 
\[F^k = \{ \alpha \in \Omega^k (M) \ | \ \forall \omega_U \text{ with } \p = \ker \omega_U \ ,\ \alpha \wedge \omega_U = \alpha \wedge \mathrm{d} \omega_U = 0\}.\]
Moreover, let $E^k := \faktor{\Omega^k(M)}{I^k}$, then it is easy to see that
\[E^0 = C^\infty (M), \ E^1 = \faktor{\Omega^1(M)}{<\omega_U>}, \ E^k = \{ 0\} \text{ if $k \ge 2$},\]
\[ F^2 = \{ \alpha \in \Omega^2 (M) \ | \ \forall \omega_U \text{ with } \p = \ker \omega_U\text{ , } \omega_U \wedge \alpha = 0\}, \ F^3 = \Omega ^3 (M) \text{ and }\]
\[F^k = \{ 0\} \text{ if $k\le 1$.}\]
We have then two natural differential operators 
\[ \mathrm{d}_\text{\p} = pr_{E_1} \circ \mathrm{d} : E^0 \longrightarrow E^1 \text{ and } \mathrm{d}_\text{\p} = \mathrm{d}_{|F^2} : F^2 \longrightarrow F^3\]
where $pr_{E_1}$ is the canonical projection from $\Omega^1 (M)$ to $E^1$.
In order to get a differential complex, one needs to define an operator $\mathrm{D} : E^1 \longrightarrow F^2$ such that $d_\text{\p} \circ \mathrm{D} : E^1 \longrightarrow F^3$ and $\mathrm{D} \circ d_\text{\p} : E^0 \longrightarrow F^2$ vanish. Rumin in \cite{Rum} defines $\Dr : E^1 \longrightarrow F^2$ in the following way.

\begin{lemma}
If $\alpha \in E^1$, there is a unique representative $\gamma \in \Omega^1 (M)$ of $\alpha$ such that $\mathrm{d} \gamma \in F^2$, then,
\[ \mathrm{D} \alpha = \mathrm{d}\gamma.\]
\end{lemma}

\begin{proof}
If $\beta \in \Omega ^1$ represents $\alpha$, locally, there is a unique function $f_U \in C^{\infty} (U)$ such that
\[ \omega_U \wedge \mathrm{d}(\beta + f_U \omega_U) = 0\]
Indeed, if $f_U \in C^\infty (U)$, then
\[\omega_U \wedge \mathrm{d}(\beta + f_U \omega_U) = \omega_U \wedge (\mathrm{d} \beta + \mathrm{d} f_U \wedge \omega_U + f_U \mathrm{d} \omega_U) = \omega_U \wedge (\mathrm{d}\beta + f_U \mathrm{d}\omega_U).\]
Thus,
\[ \omega_U \wedge \mathrm{d}(\beta + f_U \omega_U) = 0 \iff (\mathrm{d}\beta + f_U \mathrm{d}\omega_U)_{|\text{\p}} = 0.\]
Since $(\mathrm{d}\omega_U)_{|\text{\p}}$ is nondegenerate, there is a unique such function.
\newline
So, there is a unique $\widetilde \alpha_U \in \Omega^1 (U)$ representing $\alpha$ on $U$ and such that $\mathrm{d} \widetilde \alpha_U \in F^2 (U)$. Uniqueness ensures that these $\widetilde \alpha_U$ define a global $1$-form $\alpha \in \Omega^1 (M)$ representing $\alpha$ and such that $\mathrm{d} \widetilde \alpha \in F^2 (M)$.
\end{proof}

$\mathrm{D}$ being defined, we just have to verify that $\mathrm{d}_\text{\p} \circ \mathrm{D}$ and $\mathrm{D} \circ \mathrm{d}_\text{\p}$ vanish.
\newline
If $\alpha \in E^1$, let $\gamma \in \Omega^1$ be the unique representative of $\alpha$ such that $\mathrm{d} \gamma \in F^2$. Then,
\[ (\mathrm{d}_\text{\p} \circ \mathrm{D})\alpha = (\mathrm{d} \circ \mathrm{d}) \gamma = 0.\]
If $f\in C^{\infty} (M)$, by definition of $\mathrm{d}_\text{\p}$, $\mathrm{d} f$ is a representative of $\mathrm{d}_\text{\p} f$ and $\mathrm{d} (\mathrm{d}f) = 0 \in F^2$. By definition of $\Dr$,
\[ (\mathrm{D} \circ \mathrm{d}_\text{\p} )f = \mathrm{d} (\mathrm{d} f) = 0.\]

\begin{definition}

Rumin's complex is the differential complex
\[ 0 \longrightarrow \R \longrightarrow E^0 \overset{\mathrm{d}_\text{\p}} \longrightarrow E^1 \overset{\mathrm{D} }\longrightarrow F^2 \overset{\mathrm{d}_\text{\p}} \longrightarrow F^3 \longrightarrow 0.\]
We write $H^{k}_\ru (M)$ the cohomology spaces corresponding to this complex.

\end{definition}

It is proved in \cite{Rum} that this complex is locally exact and so it is an acyclic resolution of the constant sheaf. Its cohomology spaces are then isomorphic to the corresponding ones in de Rham cohomology. In the following, we will denote $H^\bullet_\dr (M)$ the de Rham cohomology.

\begin{remark}\label{rk1}
The spaces $E^k$ and $F^k$ can be interpreted as spaces of sections of vector bundles. Let $A$ be the subbundle of $T^\ast M$ of forms vanishing on the contact distribution $\p$. Then, $E^1$ is the space of sections of the bundle
\[\wedge ^1 = \faktor{T^\ast M}{A} \simeq \p ^\ast. \]
Let $\wedge ^2$ be the subbundle of $\wedge ^2 T^\ast M$ of forms vanishing on $\p ^2$, then $F^2$ is the space of sections of $\wedge ^2$.
\end{remark}

\subsection{Splitting of Rumin complex on strictly pseudoconvex CR manifolds}
In \cite{GL}, the authors gave a decomposition of the Rumin complex on a strictly pseudoconvex CR manifold. We emphase that the decomposition of the spaces $E^k$, $F^k$ and of the operators $\Dp$ and $\Dr$ appeared in Rumin's work \cite{Rum}.

\begin{definition}
A CR structure on a $2n+1$ dimensional manifold $M$ is a rank $n$ subbundle $V \subset \C TM$ such that
\[ V \cap \overline V = \{ 0\} \text{ and } [C^\infty (U,V_{|U}), C^\infty (U,V_{|U})] \subset C^\infty (U,V_{|U})\]
for every open set $U$ in $M$.
\newline
If $M$ is a $3$-dimensional manifold, the second condition is automatically verified for every rank $1$ subbundle $V \subset \C TM$.
\newline
A strictly pseudoconvex CR manifold is a CR manifold $(M,V)$ such that $V\oplus \overline V$ is a (complex) contact structure on $M$.
\end{definition}

Let now $(M,V)$ be a $3$-dimensional strictly pseudoconvex CR manifold. Then, $\p = V\oplus \overline V$ is a contact structure on $M$ and we can consider Rumin complex on $M$. In order to decompose the Rumin complex, a first step is to decompose the spaces $E^k$ and $F^k$. Let $p, q$ be non-negative integers, then define
\[ E^{p,q} = \{ \alpha \in E^{p+q} \ | \ \exists \gamma \in \Omega ^{p+q} \text{ s.t. } pr_{E_1} (\gamma) = \alpha \text{ and } \gamma_{|V\oplus \overline V} \text{ is of type $(p, q)$}\}\]
and 
\[ F^{p,q} = \{ \gamma \in F^{p+q} \ | \ (\iota_X \gamma)_{|V\oplus \overline V} \text{ is of type $(p-1, q)$ } \forall X \notin V\oplus \overline V \}\]
with the standard convention that a $(p,q)$-form with $p<0$ or $q<0$ is $0$. The only non-trivial spaces are
\[ E^{0,0} = C^{\infty} (M, \C), \ E^{1,0}, \ E^{0,1}, \ F^{2,0}, \ F^{1,1} \text{ and } F^{2,1} = \Omega^3 (M, \C).\]
It is also clear that $\C E^k = \oplus_{p+q = k} E^{p,q}$ and $\C F^k = \oplus_{p+q = k} F^{p,q}$ and, for $p+q \neq 1$, $\mathrm{d}_\p$ induces
\[ \mathrm{d}_\p : E^{p,q} \longrightarrow E^{p+1, q}\oplus E^{p, q+1} \text{ and } \mathrm{d}_\p : F^{p,q} \longrightarrow F^{p+1, q} \oplus F^{p, q+1}\]
which allows us to decompose $\mathrm{d}_\p$ in two operators $\mathrm{d}' : R^{p,q} \longrightarrow R^{p+1, q}$ and $\mathrm{d}'' : R^{p,q} \longrightarrow R^{p,q+1}$ where $R^{p,q} = E^{p,q}$ if $p+q \le 1$ and $F^{p,q}$ if not.
\newline
When $p+q = 1$, things are not as easy. In general,
\[ \Dr : E^{p,q} \longrightarrow F^{p+1, q} \oplus F^{p,q+1} \oplus F^{p+2, q-1},\]
which gives a decomposition in three operators
\[ \Dr ' : E^{p,q} \longrightarrow F^{p+1, q} , \ \Dr '' : E^{p,q} \longrightarrow F^{p, q+1} \text{ and } \Dr^+ : E^{p,q} \longrightarrow F^{p+2 , q-1}.\]
To be precise,
\[ \Dr : E^{1,0} \longrightarrow F^{2,0} \oplus F^{1,1}\]
so that $\Dr^+$ vanishes on $E^{1,0}$ and
\[ \Dr : E^{0,1} \longrightarrow F^{1,1} \oplus F^{2,0}\]
which means that $\Dr ''$ vanishes on $E^{0,1}$. Relations $\Dp \circ \mathrm{D} = \mathrm{D} \circ \Dp = 0$ and the decomposition according to the bidegree give the following:
\[\mathrm{d}''\Dr''=\Dr''\mathrm{d}'' = 0,\]
\[\mathrm{d}'\Dr'' + \mathrm{d}''\Dr' = \Dr'\mathrm{d}'' + \Dr'' \mathrm{d}' = 0,\]
\[\mathrm{d}'\Dr^+ = \Dr^+\mathrm{d}' = 0,\]
\[\mathrm{d}'\Dr'+\mathrm{d}''\Dr^+ = \Dr' \mathrm{d} ' + \Dr^+ \mathrm{d}'' = 0\]
\begin{remark}
We don't obtain a double complex since $\mathrm{d}'\Dr'$ and $\Dr' \mathrm{d}'$ don't vanish.
\end{remark}

\begin{definition}
The Garfield-Lee complexes are given by the operators $\mathrm{d}''$ and $\Dr''$. That is, the three complexes
\[ 0\longrightarrow\text{CR ($M$)} \longrightarrow E^{0,0} \overset{\mathrm{d}''} \longrightarrow E^{0,1} \overset{\Dr ''} \longrightarrow 0\]
\[ 0 \longrightarrow \ker \left( \Dr'' : E^{1,0} \longrightarrow F^{1,1} \right) \longrightarrow E^{1,0} \overset{\Dr ''} \longrightarrow F^{1,1} \overset{\mathrm{d}''} \longrightarrow 0\]
\[ 0\longrightarrow \ker \left( \mathrm{d} '' : F^{2,0} \longrightarrow F^{2,1} \right) \longrightarrow F^{2,0} \overset{\mathrm{d}''} \longrightarrow F^{2,1} \longrightarrow 0.\]
Their cohomology spaces are denoted $H^{p,q}_{G-L} (M)$.
\end{definition}

It is clear that the $H^{p,q}_{G-L}$ are CR invariants of the manifold. Moreover, Theorem 1 in \cite{GL} states that the $H^{p,q}_{G-L}$ are isomorphic to their corresponding spaces in the $\overline \partial _b$-complex.

\begin{remark}\label{rk2}
As for Rumin complex, we can see spaces $E^{p,q}$ and $F^{p,q}$ as spaces of sections of complex line bundles. Let $B^{1,0}$ (resp. $B^{0,1}$) be the subbundle of $\C T^\ast M$ of forms vanishing on $\overline V$ (resp. the subbundle of $\C T^\ast M$ of forms vanishing on $V$). Then, $E^{1,0}$ is the space of sections of
\[ \wedge^{1,0} = \faktor{B^{1,0}}{A} \simeq V^\ast\]
and $E^{0,1}$ is the space of sections of
\[ \wedge^{0,1} = \faktor{B^{0,1}}{A} \simeq \overline V^\ast.\]
Let $\wedge ^{1,1}$ (resp. $\wedge ^{2,0}$) be the subbundle of $\wedge ^2$ of forms with vanishing interior product with every vector of $V$ (resp. the subbundle of $\wedge ^2$ of forms with vanishing interior product with every vector of $\overline V$). Then $F^{1,1}$ is the space of sections of $\wedge ^{1,1}$ and $F^{2,0}$ is the space of sections of $\wedge ^{2,0}$.
\end{remark}

\subsection[Another differential complex]{Another differential complex, definitions and computations}
\subsubsection*{Definition of the complex}
Let us introduce a new differential complex on strictly pseudoconvex $3$-dimensional CR manifolds heavily inspired by Rumin complex and its splitting. Let $(M,V)$ be a compact strictly pseudoconvex $3$-dimensional CR manifold. We consider again spaces $E^{p,q}$, $F^{p,q}$ and operators $\Dp, \mathrm{d} ', \mathrm{d} '', \Dr, \Dr ', \Dr ''$ and $\Dr^+$. The complex is:
\[ 0 \longrightarrow \C \longrightarrow \text{CR} (M) \overset{\Dp} \longrightarrow E^{1,0} \overset{\Dr} \longrightarrow F^2 \overset{\Dp} \longrightarrow F^3 \longrightarrow 0.\]
where $\text{CR}(M)$ is the space of CR functions on $M$. In the case of compact Riemann surfaces, the space of holomorphic forms is the space of $\mathrm{d}$-closed $(1,0)$-forms. Since there are no nonconstant holomorphic functions on compact Riemann surfaces, there are no $\mathrm{d}$-exact $(1,0)$-forms except $0$. Meaning that, the space of holomorphic forms on a compact Riemann surface identifies with the space
\[\faktor{ \{ \text{closed }(1,0)\text{-form} \}}{\{ \text{exact }(1,0)\text{-form}\}}.\]
Thus, we will be interested in the space:
\[ H^{1,0} (M) = \faktor{\ker \left(\Dr _{|E^{1,0}} : E^{1,0} \longrightarrow F^2\right)}{\mathrm{Im} \left( \mathrm{d}_{\p_{|\mathrm{CR}(M)}} : \mathrm{CR} (M) \longrightarrow E^{1,0}\right)}.\]
Since the space $H^{1,0} (M)$ takes into account only $(1,0)$-forms, we can notice that $\mathrm{Im} \left( \mathrm{d}_{\p_{|\mathrm{CR}(M)}} : \mathrm{CR} (M) \longrightarrow E^{1,0}\right) = E^{1, 0} \cap \mathrm{Im} \left(\Dp : C^{\infty} (M) \longrightarrow E^1\right)$ so that
\begin{eqnarray}\label{eq3.1}
H^{1,0} (M) = \faktor{E^{1,0} \cap \ker \left( \Dr : E^1 \longrightarrow F^2\right)}{E^{1, 0} \cap \mathrm{Im} \left(\Dp : C^{\infty} (M) \longrightarrow E^1\right)}.
\end{eqnarray}

The advantage of the space $H^{1,0} (M)$ is that it is finite dimensional.

\begin{proposition}
Let $(M,V)$ be a strictly pseudoconvex $3$-dimensional CR manifold, then
\[ \dim_\C \left( H^{1,0} (M)\right) \le \dim_\C (H^{1}_{{dR}} (M,\C)).\]
\end{proposition}

\begin{proof}

The restriction to $\ker (\Dr : E^1 \longrightarrow F^2) \cap E^{1,0}$ of the canonical projection from $\ker (\Dr : E^1 \longrightarrow F^2)$ to $H^{1}_R (M,\C)$ has kernel $\mathrm{Im} (\mathrm{d}_\p) \cap E^{1,0}$ so that it induces an injective morphism from $H^{1,0}$ to $H^1_R (M,\C)$. Since $H^1_R (M,\C)$ is isomorphic to $H^{1}_{{dR}} (M,\C)$ we get the result.

\end{proof}

It is easy to see that Rumin's operators $\Dp$ and $\Dr$ are natural (that is, for every contact transform $f : (M, \p_M) \longrightarrow (N,\p_N)$, $f^\ast \circ \Dr = \Dr \circ f^\ast$ and $f^\ast \circ \Dp = \Dp \circ f^\ast$). Moreover, a CR diffeomorphism preserves the bidegree. Therefore, $H^{1,0}$ is a CR invariant :

\begin{proposition}

Let $(M,V)$ and $(N,W)$ be strictly pseudoconvex $3$-dimensional CR manifolds and $f : M \longrightarrow N$ be a CR diffeomorphism. Then $f^\ast$ induces an isomorphism from $H^{1,0} (N)$ on $H^{1,0} (M)$.

\end{proposition}

\subsubsection*{Circle bundles over surfaces}

An important class of $3$-dimensional manifolds is formed by circle bundles over surfaces. Let $(M,V)$ be a strictly pseudoconvex $3$-dimensional CR compact manifold endowed with a free action of $\mathrm{U}(1)$. The quotient $\Sigma = \faktor {M}{\mathrm{U}(1)}$ is a compact surface of genus $g$ and we denote $\Pi : M \longrightarrow \Sigma$ the projection. Since $\Pi$ is a submersion, it induces an injective morphism $\Pi ^\ast : H^{1}_{dR} (\Sigma, \C) \longrightarrow H^1 _{dR} (M, \C)$ and, since $pr_{E_1} : \Omega ^\bullet (M) \longrightarrow E^1$ induces an isomorphism from $H^1_{dR} (M)$ to $H^1_R (M)$,  $\overline \Pi ^\ast = pr_{E_1} \circ \Pi^\ast : \Omega^1 (\Sigma) \longrightarrow E^1$ also induces an injective morphism $\phi : H^1_{dR} (\Sigma) \longrightarrow H^1_R (M)$. First, we recall the following well-known for completeness sake.

\begin{lemma}
\label{lem3.2.4}
Let $N$ be a circle bundle over a compact surface $S$ of genus $g$ and let $e$ be the Euler class of the bundle. Then
\[ \begin{cases}
\dim (H^1_{dR} (N)) = 2g, \text{ if $e\neq0$}\\ 
\dim (H^1_{dR} (N)) = 2g +1, \text{ if $e=0$}
\end{cases}.\]

\end{lemma}

\begin{proof}

Gysin long exact sequence reads here as:
\[
 0\rightarrow H^1_{dR} (S) \rightarrow H^1_{dR} (N) \rightarrow H^0_{dR} (S) \overset{e_\wedge}\rightarrow H^2_{dR} (S) \rightarrow H^2_{dR} (N) \rightarrow H^1_{dR} (S)\rightarrow 0
\]
where $e_\wedge$ is the exterior product by $e$. If $e \neq 0$, then $e_\wedge$ is an isomorphism. So, we have the following exact sequence:
\[ 0 \longrightarrow H^1_{dR} (S) \longrightarrow H^1_{dR} (N) \longrightarrow H^0_{dR} (S) \longrightarrow H^2_{dR} (S) \longrightarrow 0.\]
If $e = 0$, then we have:
\[ 0 \longrightarrow H^1_{dR} (S) \longrightarrow H^1_{dR} (N) \longrightarrow H^0_{dR} (S) \longrightarrow 0.\]
Since we know $H^k _{dR} (S)$, the result follows easily.

\end{proof}

This lemma has as important consequence that $\Pi ^\ast : H^{1}_{dR} (\Sigma, \C) \longrightarrow H^1 _{dR} (M, \C)$ is in fact an isomorphism as well as $\phi : H^1_{dR} (\Sigma) \longrightarrow H^1_R (M)$ when the circle bundle has non-vanishing Euler class. Moreover, we saw that we can see $H^{1,0}(M)$ as a subset of $H^1 _R (M, \C)$. Consequently, in the case of a circle bundle with non vanishing Euler class, $H^{1,0} (M)$ is isomorphic to a subspace of $H^1_{dR} (\Sigma, \C)$.
\newline

We now compute explicitly $H^{1,0}$ for compact CR circle bundles. Let $R$ be the global vector field induced by the $\mathrm{U}(1)$-action on the CR manifold $(M,V)$. We say that the $\mathrm{U}(1)$-action is transversal if for each $x\in M$
\[ \C T_x M = \C R_x \oplus V_x \oplus \overline V_x\]
and the action is said CR if
\[ [R, \Gamma (V)] \subset \Gamma (V)\]
where $\Gamma (V)$ refers to the space of sections of $V$. A {\it CR circle bundle} is a strictly pseudoconvex CR manifold endowed with a free, transversal, CR action of $\mathrm{U}(1)$. Let $(M,V)$ be a CR circle bundle, $R$ the vector field induced by the $\mathrm{U} (1)$ action and denote $\omega$ its dual form (that is, $\omega (R) = 1$ and $\omega_{|V\oplus \overline V} = 0$). Then, $\omega$ is a (global) contact form on $M$ and $R$ is its associated Reeb vector field. For the computation of $H^{1,0} (M)$ we have to recall a few facts about pseudohermitian manifolds.

A pseudohermitian structure on a manifold $N$ is a pair $(\omega , J)$ where $\omega$ is a contact form on $M$ and $(\p=\ker (\omega) , J)$ is a CR structure on $N$. From the discussion above, a CR circle bundle admits a natural pseudohermitian structure. One can define a natural connection on a pseudohermitian manifold: the Tanaka-Webser connection (see \cite{Tana, Web}).

\begin{definition}
Let $(N,\omega , J)$ be a pseudohermitian manifold. Then, there is a unique connection $\nabla : \Gamma (TM) \longrightarrow \Gamma (T^\ast M \otimes TM)$ such that
\begin{enumerate}
\item{for each $X\in \Gamma (TM)$, $\nabla_X \left( \Gamma(\p)\right) \subset \Gamma (\p)$,}
\item{$\nabla R = 0$, $\nabla J = 0$, $\nabla \mathrm{d} \omega = 0$,}
\item{its torsion $T$ satisfies 
\[ \forall X,Y \in \Gamma (\p), \ T(X,Y) = \mathrm{d} \omega (X,Y) R \text{ and}\]
\[ \forall X \in \Gamma (TM), \ T(R,JX) + JT(R,X) = 0.\]}
\end{enumerate}
$T(R,.)$ is called the pseudohermitian torsion of the manifold.
\end{definition}

Webster \cite[p.~33]{Web} proved that the pseudohermitian torsion vanishes if and only if the Reeb vector field induces a $1$-parameter family of CR automorphisms. In particular, a CR circle bundle has a vanishing pseudohermitian torsion.
The Webster metric $g$ on a pseudohermitian manifold is defined by
\[g(X,Y) = \mathrm{d} \omega (X,JY), \ g(X,R)=g(R,X) = 0, \ g(R,R) =1\]
for every $X,Y \in \Gamma (\p)$. This metric canonically extends to differential forms.
Once a metric is given, we can identify $E^1$ with the orthogonal of $I^1$ and then, define the adjoint operators of $\Dr$ and $\Dp$ denoted $\Dr ^\ast$ and $\delta_\p$ respectively. Following Rumin \cite[p.~290]{Rum}, define the laplacian
\[ \Delta_\p = \Dr^\ast \Dr + \left(\Dp \delta_\p\right)^2 \text{ on $E^1$.}\]
Then, Rumin proved a Hodge decomposition. If $(N,\omega , J)$ is a compact pseudohermitian manifold, then we have an orthogonal decomposition
\[E^1=\ker\ \Delta_\p\oplus\mathrm{Im}\ \Delta_\p.\] 
Moreover,
\begin{eqnarray}\label{eqkerd}
\ker\ \Dr = \ker \ \Delta_\p \oplus \mathrm{Im} \ \Dp.
\end{eqnarray} 
He also showed (see e.g. \cite[p.~312]{Rum}) that $\Delta_\p$ preserves the bidegree when the manifold has vanishing pseudohermitian torsion which leads to:

\begin{proposition}\label{prop1.3.5}
Let $M$ be a compact CR circle bundle. Then $H^{1,0}$ is isomorphic to $E^{1,0} \cap \ker \ \Delta_\p$.
\end{proposition}

\begin{proof}

Let $\Psi : E^1 \longrightarrow \ker \ \Delta_\p$ be the orthogonal projection. First, we show that $\Psi \left(E^{1,0}\right) \subset E^{1,0}$. Take $\alpha = \beta + \Delta_\p \gamma \in E^{1,0}$ with $\Delta_\p \beta = 0$ and decompose $\beta = \beta^{1,0} + \beta^{0,1}$, $\gamma = \gamma^{1,0} + \gamma^{0,1}$ with $\beta^{p,q} , \gamma^{p,q} \in E^{p,q}$. Then, since $\alpha \in E^{1,0}$, $\Delta_\p \beta = 0$ and $\Delta_\p$ preserves the bidegree, 
\[ \beta ^{0,1} = -\Delta_\p \gamma ^{0,1} \text{ and } \Delta_\p \beta^{0,1} = 0.\]
Therefore, $\beta ^{0,1} \in \ker\ \Delta_\p\cap\mathrm{Im}\ \Delta_\p = \{0\}$.

Now, consider $\Psi_{|E^{1,0} \cap \ker \ \Dr}$. Then, using the decomposition \ref{eqkerd}, its kernel is exactly $E^{1,0} \cap \mathrm{Im} \ \Dp$. Thus $\Psi_{|E^{1,0} \cap \ker \ \Dr}$ induces an injective linear map
\[ \overline \Psi : H^{1,0} \longrightarrow E^{1,0} \cap \ker \ \Delta_\p.\]
The surjectivity of $\overline \Psi$ follows from $\ker \ \Delta_\p \subset \ker \ \Dr$.
\end{proof}

The fact that $\Delta_\p$ preserves the bidegree also induces a decomposition of $H^1_R$ for compact pseudohermitian manifolds with vanishing pseudohermitian torsion. Namely, if $(N,\omega , J)$ is a compact pseudohermitian manifold with vanishing pseudohermitian torsion, then
\[ H^1_R (N) \simeq \ker \Delta_\p = E^{1,0} \cap \ker \Delta_\p \oplus E^{0,1} \cap \ker \Delta_\p.\] 
If $M$ is a compact CR circle bundle, using Lemma \ref{lem3.2.4}, we get then,
\[ \dim\left(H^1_R (M)\right) = \dim\left(H^1_{dR} (M)\right) = 2g\]
where $g$ is the genus of $\faktor{M}{\mathrm{U}(1)}$. In conclusion, we have:

\begin{proposition}\label{prop1.3.6}
Let $M$ be a compact CR circle bundle over a surface of genus $g$. Then
\[\dim \left(H^{1,0} (M) \right)= g.\]
\end{proposition}

Even more explicitly, $\Sigma = \faktor{M}{\mathrm{U}(1)}$ inherits of a Riemann surface structure for which the projection $\Pi : M \longrightarrow \Sigma$ is a CR map and $\Pi ^\ast$ induces a natural isomorphism from $K(\Sigma)$ to $H^{1,0}$ where $K(\Sigma)$ is the space of holomorphic forms on $\Sigma$.

\section{Quadratic differentials on (spherical) CR manifolds}\label{sec2}
In this section, we introduce the notion of {\it quadratic differential} on spherical CR manifolds and discuss its first features. Before that, in the first section we describe the splitting of Rumin complex in spherical CR geometry. Then, section \ref{sec2.1} gives the general definition of a quadratic differential on a strictly pseudoconvex CR manifold, explains what it induces in spherical CR geometry and discusses the first geometric data given by a quadratic differential such as trajectories and length. After that, we introduce several "operators of holomorphy" on quadratic differentials in section \ref{sec2.2}. And finally explain how those operators extend to forms of positive degree.

\subsection{Description in spherical CR geometry}
A (3-dimensional) spherical CR manifold is a CR manifold locally equivalent to the unit sphere in $\C^2$ endowed with its standard CR structure. Another local model of spherical CR geometry is the Heisenberg group. To begin, we recall a few basic facts about it.

The Heisenberg group $\h$ is $\C\times \R$ endowed with the group law 
\[(z,t)(z',t')=(z+z',t+t'+2\Im (z\overline z')).\]
The CR structure on it is given by $V_\h = span(Z)$ where $Z$ is the left invariant vector field
\[Z=\dz+i\overline z \dt.\]
It is a strictly pseudoconvex one since $V_\h \oplus \overline V_\h = \ker (\omega)$ where $\omega = \mathrm{d}t-i\overline z \mathrm{d}z + iz\mathrm{d}\overline z$ is a contact form. Moreover, it is well-known that $(\h,V_\h)$ is CR-equivalent to the sphere $S^3\subset \C^2$ minus one point. The CR automorphism group of $S^3$ is the group of biholomorphisms of the unit ball, that is $\mathrm{PU}(2,1)$ which verifies a Liouville type theorem. 

\begin{theorem}
Let $\varphi : U \longrightarrow V$ be a CR diffeomorphism between open subsets of $S^3$, then $\varphi$ is the restriction of a CR automorphism of $S^3$.
\end{theorem}

Thus, a spherical CR manifold can be seen as a manifold endowed with a $(\mathrm{PU}(2,1), \h \cup\{\infty\})$ (or a $(\mathrm{PU}(2,1), S^3)$) structure. That is, a manifold with an atlas of charts with values in $\h$ and transition functions in $\mathrm{PU}(2,1)$. Such an atlas will be called a {\it spherical CR atlas}. Having a spherical CR atlas on a manifold will allow us to define $(p,q)$-forms with cocyle relations.

We now provide a description of forms and operators $\mathrm{d}', \mathrm{d}'', \mathrm{D}', \mathrm{D}''$ and $\mathrm{D}^+$ in open subsets of $\h$. For a general local description on strictly pseudoconvex CR manifolds, one can refer to e.g. \cite{AL}. First notice that $(Z,\overline Z, T=\dt)$ is a frame of $T\h$ with dual coframe $(\mathrm{d}z, \mathrm{d}\overline z, \omega)$ and easy computations show :
\[ [Z,\overline Z]=2iT, \ [Z,T]=[\overline Z, T] = 0, \ \mathrm{d} \omega = 2i\mathrm{d}z\wedge \mathrm{d} \overline z.\]
Let $\alpha$ be a $(p,q)$-form in an open subset $U$ of $\h$. Then, $\alpha$ is of the following form according to the value of $(p,q)$:
\[p=q=0: \ \alpha = f;\]
\[p=1,\ q=0: \ \alpha = [f\mathrm{d}z]_\omega;\]
\[p=0, \ q=1: \ \alpha = [f\mathrm{d}\overline z]_\omega;\]
\[p=q=1: \ \alpha = f \mathrm{d}\overline z \wedge \omega;\]
\[p=2, \ q=0: \ \alpha = f\mathrm{d}z\wedge \omega;\]
\[p=2,\ q=1: \ \alpha = f \mathrm{d}z\wedge\mathrm{d}\overline z \wedge \omega\]
where $f\in C^\infty(U,\C)$. Then, if $f \in C^\infty(U,\C)$ :
\[\mathrm{d}'f = [Zf\mathrm{d}z]_\omega \text{ , } \mathrm{d}'' f = [\overline Z f \mathrm{d}\overline z]_\omega,\]
\[ \mathrm{d}'(f\mathrm{d}\overline z\wedge \omega) = Zf\mathrm{d}z\wedge\mathrm{d}\overline z\wedge \omega \text{ and}\]
\[\mathrm{d}''(f\mathrm{d}z\wedge \omega) = -\overline Z f\mathrm{d}z\wedge\mathrm{d}\overline z\wedge \omega.\]
For $E^{1,0}$, let $f\mathrm{d}z + h \omega$ with $f,h \in C^\infty(U,\C)$ be a representative of $[f\mathrm{d}z]_\omega$. Then we have
\[\mathrm{d} (f\mathrm{d}z+h\omega) = \left(2ih-\overline Z f\right)\mathrm{d}z\wedge\mathrm{d}\overline z + \left(Zh-Tf\right)\mathrm{d}z\wedge\omega+ \overline Z h \mathrm{d}\overline z \wedge \omega.\]
Thus, $\mathrm{d} (f\mathrm{d}z+h\omega) \in F^2$ if and only if $2ih = \overline Z f$. Consequently,
\[ \Dr [f\mathrm{d}z]_\omega = \mathrm{d} \left(f\mathrm{d}z + \frac{1}{2i}\overline Z f \omega\right) = \left(\frac{1}{2i}Z\overline Z f - Tf\right)\mathrm{d}z\wedge\omega + \frac{1}{2i} \overline Z^2 \mathrm{d}\overline z \wedge \omega.\]
Decomposing according to the bidegree gives
\[\Dr ' [f\mathrm{d}z]_\omega = \left(\frac{1}{2i}Z\overline Z f - Tf\right)\mathrm{d}z\wedge\omega \text{ and}\]
\[\Dr '' [f\mathrm{d}z]_\omega =\frac{1}{2i} \overline Z^2 f \mathrm{d}\overline z \wedge \omega.\]
Doing the same thing for $E^{0,1}$, we find
\[\Dr' [f\mathrm{d}\overline z]_\omega = \left(\frac{1}{2i}\overline Z Zf + Tf\right)\omega\wedge\mathrm{d}\overline z \text{ and}\]
\[\Dr ^+ [f\mathrm{d}\overline z]_\omega = \frac{1}{2i} Z^2 f \omega \wedge \mathrm{d}z.\]

Now, let $M$ be a manifold endowed with a spherical CR atlas $(U_i, \varphi_i : U_i \rightarrow U_i ' \subset \h)_i\in I$. Let $\alpha$ be a $(1,0)$-form on $M$. Since the $\varphi_i$ are CR diffeomorphisms, 
\[ \alpha = \varphi_i ^\ast [f_i \mathrm{d}z]_\omega \text{ on $U_i$, where $f_i \in C^\infty (U_i ' , \C)$.}\]
If $(U_j, \varphi_j)$ is another chart, $\alpha = \varphi_j ^\ast [f_j \mathrm{d}z]_\omega \text{ on $U_j$, with $f_j \in C^\infty (U_i ' , \C)$}$ and so
\[[f_i \mathrm{d}z]_\omega = [(f_j\circ g_{j,i})Zg_{j,i}^1 \mathrm{d}z + (f_j\circ g_{j,i})Tg_{j,i}^1 \omega]_\omega\]
on $\varphi_i (U_i \cap U_j)$ where $g_{j,i} (g_{j,i}^1, g_{j,i}^2)= \varphi_j \circ \varphi_i ^{-1}$. Consequently :
\[f_i = (f_j\circ g_{j,i})Zg_{j,i}^1 \text{ on $\varphi_i (U_i \cap U_j)$.}\]
Doing the same thing for $(0,1)$-forms on $M$, we can provide the following definition of $(p,q)$-forms ($p+q = 1$).

\begin{definition}

A $(1,0)$-form on a spherical CR manifold $M$ is a collection of functions $f_i : U_i ' \longrightarrow \C$, where $(U_i, \varphi_i : U_i \longrightarrow U_i ')_i$ is an atlas of $M$, such that
\[ \left( f_j \circ g_{j,i} \right) Zg_{j,i} ^1 = f_i \text{ on $\varphi_i (U_i  \cap U_j) $, where $g_{j,i} = \varphi_j \circ \varphi_i ^{-1}$}.\]
A $(0,1)$-form on a spherical CR manifold $M$ is a collection of functions $f_i : U_i ' \longrightarrow \C$, where $(U_i, \varphi_i : U_i \longrightarrow U_i ')_i$ is an atlas of $M$, such that
\[ \left( f_j \circ g_{j,i} \right) \overline{Zg_{j,i} ^1} = f_i \text{ on $\varphi_i (U_i  \cap U_j) $, where $g_{j,i} = \varphi_j \circ \varphi_i ^{-1}$}.\]
\end{definition}

Since the operators $\Dr ', \Dr ''$ and $\Dr ^+$ commute with CR maps, we can define:

\begin{definition}
Let $M$ be a spherical CR manifold.
\begin{enumerate}
\item{Let $\alpha = (f_i)_i$ be a $(1,0)$-form on M. Then $\alpha$ is said:
\begin{itemize}
\item{$D'$-closed if for every $i$, 
\[Z\overline Z f_i = 2i Tf_i,\]}
\item{$D''$-closed if for every $i$,
\[\overline Z^2 f_i = 0.\]}
\end{itemize}}
\item{Let $\alpha = (f_i)_i$ be a $(0,1)$-form on M. Then $\alpha$ is said:
\begin{itemize}
\item{$D'$-closed if for every $i$, 
\[i\overline Z Z f_i = 2 Tf_i,\]}
\item{$D^+$-closed if for every $i$,
\[ Z^2 f_i = 0.\]}
\end{itemize}}
\end{enumerate}
\end{definition}

\subsection{General definition and geometry induced by a quadratic differential}\label{sec2.1}
Remark \ref{rk2} indicated that we can see $(1,0)$-forms on a strictly pseudoconvex CR manifold as smooth sections of the line bundle $\wedge^{1,0}$ which is defined as the quotient of the space of forms vanishing on the conjugate of the CR distribution by the space of forms vanishing on the contact distribution. The remark also indicated that this line bundle is naturally isomorphic to the dual of the CR distribution. It is then tempting to see $\wedge^{1,0}$ as analogue to the canonical bundle of a Riemann surface. With that in mind, the definition of a quadratic differential is natural.

\begin{definition}[Quadratic differential]
A quadratic differential on a strictly pseudoconvex CR manifold is a smooth section of the line bundle $\wedge^{1,0} \otimes \wedge^{1,0}$.
\end{definition}

As for $(1,0)$-forms, one can understand quadratic differentials on a spherical CR manifold in terms of cocycle relations. Let $M$ be a manifold endowed with a spherical CR atlas $(U_i, \varphi_i : U_i \rightarrow U_i ' \subset \h)_{i\in I}$. On every open subset of $\h$, a quadratic differential is the class modulo the contact form of a form of type $f\mathrm{d}z^2$. We write such classes $[f\mathrm{d}z^2]_\omega$. Then, if $q$ is a quadratic differential on $M$, in every chart $(U_i,\varphi_i)$ we have
\[q =\varphi_i ^\ast[q_i\mathrm{d}z^2]_\omega\]
with $q_i \in C^\infty(U_i ', \C)$. Since $q$ is globally defined on $M$, if $U_i\cap U_j \neq \emptyset$ we must have
\[ [q_i\mathrm{d}z^2]_\omega = g_{j,i}^\ast [q_j\mathrm{d}z^2]_\omega \text{ on $\varphi_i(U_i\cap U_j)$ where $g_{j,i} = \varphi_j\circ\varphi_i^{-1}$.}\]
That is
\[[q_i\mathrm{d}z^2]_\omega = [(q_j\circ g_{j,i})\left(Zg_{j,i}^1\right)^2\mathrm{d}z^2]_\omega \text{ on $\varphi_i(U_i\cap U_j)$ where $g_{j,i} = \varphi_j\circ\varphi_i^{-1}$}\]
which legitimates the following definition.

\begin{definition}[Quadratic differential on a spherical CR manifold] 
A quadratic differential on a spherical CR manifold $M$ with spherical CR atlas $(U_i, \varphi_i)_{i\in I}$ is a collection of smooth complex valued functions $q_i \in C^{\infty}(\varphi_i (U_i))$ satisfying for every $i,j \in I$
\[q_i = (q_j\circ g_{j,i})\left(Zg_{j,i}^1\right)^2 \text{ on $\varphi_i(U_i\cap U_j)$ where $g_{j,i} = \varphi_j\circ\varphi_i^{-1}$.}\]
\end{definition}

The first geometric data induced by a quadratic differential on a strictly pseudoconvex manifold can be defined by analogy with the Riemann surface case. However, since quadratic differentials are defined up to a contact form, we shall restrict to {\it Legendrian curves} (that is curves everywhere tangent to the contact distribution).

\begin{definition}[Trajectories, length] 
Let $(M,V)$ be a strictly pseudoconvex CR manifold and $q$ be a (non-zero) quadratic differential on $M$. Let also $\gamma : I \rightarrow M$ be a Legendrian parametrized curve on $M$
\begin{enumerate}
\item $\gamma$ is called a
\begin{itemize}
\item horizontal trajectory of $q$ if $q(\gamma '(s)) > 0$ for all $s \in I$;
\item vertical trajectory of $q$ if $q(\gamma ' (s)) < 0$ for all $s\in I$.
\end{itemize}
We denote $\mathrm{Hor}_q$ and $\mathrm{Ver}_q$ the sets of horizontal and vertical trajectories of $q$.

\item $\sqrt{|q|}$ is a length element that can be integrated along any legendrian curve. Thus, we call the $q$-length of the curve $\gamma$, the number
\[ l_q (\gamma) = \int_\gamma \sqrt{|q|}.\]

\end{enumerate}
\end{definition}

We will discuss about known examples of maps dilating trajectories of quadratic differentials in section \ref{sec3} but we define here what it means to dilate trajectories of quadratic differentials.

\begin{definition}[Contact map preserving/dilating trajectories of quadratic differentials]
Let $(M,V)$ and $(M',V')$ be strictly pseudoconvex CR manifolds and $f:M\rightarrow M'$ a contact map. Let $q$ (resp. $q'$) be a quadratic differential on $(M,V)$ (resp. on $(M',V')$). We say that $f$ :
\begin{enumerate}
\item preserves $(\mathrm{Hor}_q, \mathrm{Hor}_{q'})$ if $f(\mathrm{Hor}_q) \subset \mathrm{Hor}_{q'}$;
\item dilates $(\mathrm{Hor}_q, \mathrm{Hor}_{q'})$ if $f$ preserves $(\mathrm{Hor}_q, \mathrm{Hor}_{q'})$ and if there is a positive number $\lambda$ such that
\[l_{q'} (f(\delta)) = \lambda l_q (\delta) \text{ for all $\delta \in \mathrm{Hor}_q$.}\]
\end{enumerate}
\end{definition}

\subsection{Operators}\label{sec2.2}

We introduce here several operators on quadratic differentials on a spherical CR manifold. One of the main features of holomorphic quadratic differentials on a Riemann surface is the existence of {\it natural coordinates} : a non-zero holomorphic quadratic differential on a Riemann surface $q$ is, away from a critical point, locally given by $\mathrm{d}w^2$ for a holomorphic coordinate $w$. 

Let now $q$ be a non-zero quadratic differential on a spherical CR manifold $M$. We wish to find conditions on $q$ under which around any point $p$ where $q(p) \neq 0$ there is a chart $(U,\varphi)$ such that 
\[ q = \varphi^\ast[\mathrm{d}z^2]_\omega.\]
We will define $3$ operators $\Dr'_2$, $\Dr''_2$ and $B_2$ on quadratic differentials such that the following holds.

\begin{proposition}\label{th221}
Let $q$ be a non-zero quadratic differential on a spherical CR manifold $M$ such that $\Dr' _2 q = 0$, $\Dr''_2 q=0$ and $B_2 q =0$. Then, around every point $p\in M$ where $q(p)\neq 0$, there is a chart $(U,\varphi)$ such that $q=\varphi^\ast [\mathrm{d}z^2]_\omega$ on $U$. 

Moreover, two such charts differ only by translation and sign, that is by $(z,t) \mapsto \left(\pm z+z_0, t+t_0 + 2\Im(\pm z \overline z_0)\right)$. 
\end{proposition}

Let $q$ be a quadratic differential on a spherical CR manifold $M$. As stated, we want $q$ to be locally of the form $\varphi^\ast[\mathrm{d}z^2]_\omega$ for a chart $(U,\varphi)$. It means that we want $q = \left(\Dp \varphi_1\right)^2$ modulo the contact form $\varphi^\ast \omega$ on $U$ where $\varphi = (\varphi_1, \varphi_2) : U\rightarrow U'\subset\h$ is a CR diffeomorphism. Thus, we first want to express conditions for a quadratic differential to be locally the square of the differential of a CR function or equivalently the pull-back by a CR function of the quadratic form $\mathrm{d}w^2$ in $\C$. 

The corresponding question for $(1,0)$-forms can be answered by the annulment of the operator $\Dr$. More precisely, if $\alpha \in E^{1,0}$ is in the kernel of $\Dr$, then, since Rumin complex is locally exact, there is locally a function $f$ such that $\alpha = \Dp f$. Moreover, since $\alpha \in E^{1,0}$, $f$ must be a CR function. Consequently, the first step is to find an operator $\Dr_2$ (split in two operators $\Dr_2 '$ and $\Dr_2 ''$), that extends $\Dr$ to quadratic differentials.

To construct these operators $\Dr_2'$ and $\Dr_2 ''$, we first define them on quadratic differentials on open subsets of the Heisenberg group. Let $[q\mathrm{d}z^2]_\omega$ be a quadratic differential on an open subset $U$ of $\h$. Define :
\[ \Dr '_{2,\h} [q\mathrm{d}z^2]_\omega = \left(2qZ\overline Z q - Zq\overline Zq -4iqTq\right)[\mathrm{d}z^3]_\omega \otimes \omega\wedge\mathrm{d}z \text{ and}\]
\[\Dr ''_{2,\h} [q\mathrm{d}z^2]_\omega = \left(2q\overline Z ^2q - \left(\overline Z q\right)^2\right)[\mathrm{d}z^3]_\omega \otimes \omega\wedge \mathrm{d}\overline z.\]

\begin{remark}\label{rk3}
The idea behind these operators is to apply $\Dr$ to the square root of a quadratic differential $q$. Since $\Dr$ is a second order operator, we can multiply "$\Dr \sqrt{q}$" by "$q^{\frac{3}{2}}$" to obtain an operator acting on $q$ and not on its square root. After scale of constants and splitting in two, it gives $\Dr '_{2,\h}$ and $\Dr''_{2,\h}$.

An idea of algebraic definition would be to define $\Dr'_2 (\alpha\otimes\beta) = \alpha\otimes \beta\otimes \alpha \otimes \Dr' \beta$ on pure tensors. However, one cannot immediately extend this to quadratic differentials since it is not a linear operator but a quadratic one. Thus, one needs to define the bilinear operator associated to $\Dr '_2$ in order to extend the definition from pure tensors to quadratic differentials.
\end{remark}

We can now use $\Dr'_{2,\h}$ and $\Dr''_{2,\h}$ to define the following.

\begin{definition}\label{def223}
Let $M$ be a spherical CR manifold $M$ with atlas $(U_i, \varphi_i)_i$. Then the following operators are well-defined (where $\Gamma (E)$ refers to the space of smooth sections of any line bundle $E$ over $M$ and $\wedge^{p,q}$ for $p+q = 2$ are those from remark \ref{rk2}):
\[\begin{array}{ccccc}
\Dr_2 ' : & \Gamma\left(\left(\wedge^{1,0}\right)^{\otimes 2}\right) & \longrightarrow & \Gamma\left(\left(\wedge^{1,0}\right)^{\otimes 3} \otimes \wedge^{2,0}\right) & \\
& q & \longmapsto & \varphi_i^\ast\Dr'_{2,\h}[q_i\mathrm{d}z^2]_\omega & \text{ on $U_i$ where $q=\varphi_i ^\ast[q_i\mathrm{d}z^2]_\omega$}
\end{array}\]
and 
\[\begin{array}{ccccc}
\Dr_2 '' : & \Gamma\left(\left(\wedge^{1,0}\right)^{\otimes 2}\right) & \longrightarrow & \Gamma\left(\left(\wedge^{1,0}\right)^{\otimes 3} \otimes \wedge^{1,1}\right) & \\
& q & \longmapsto & \varphi_i^\ast\Dr''_{2,\h}[q_i\mathrm{d}z^2]_\omega & \text{ on $U_i$ where $q=\varphi_i ^\ast[q_i\mathrm{d}z^2]_\omega$.}
\end{array}\]
\end{definition}

\begin{proof}
We have to prove that this definition doesn't depend on the choice of a chart $(U_i,\varphi_i)$. For that, it is enough to show that $\Dr_{2,\h} '$ and $\Dr_{2,\h} ''$ commute to all CR diffeomorphisms between open sets of the Heisenberg group. We will prove it only for $\Dr_{2,\h} ''$, the other one being proved the same way. So, let $[q\mathrm{d}z^2]_\omega$ be a quadratic differential on an open subset $U$ of $\h$ and $g=(g_1,g_2) : U\rightarrow V$ be a CR diffeomorphism. We begin with establishing several formulas for the derivatives of $g_1$. Since $g$ is a CR diffeomorphism, $g_1$ is a CR function. From that, we deduce
\begin{eqnarray}\label{eq3}
\overline Z Z g_1 & = &2iTg_1.
\end{eqnarray}
Moreover, $g_2 + i|g_1|^2$ is also a CR function. It leads to
\[ \overline Z g_2= -ig_1\overline{Zg_1} \text{ and, since $g_2$ is real valued, } Zg_2 = i\overline g_1 Zg_1,\]
which gives, using $[\overline Z, Z]=2iT$, 
\[ |Zg_1|^2 = Tg_2 -i\overline g_1 Tg_1 +ig_1 \overline{Tg_1}.\]
Thus, we have
\begin{eqnarray}\label{eq4}
\overline Z\left(|Zg_1|^2\right) = T\overline Z g_2 - i Tg_1\overline {Zg_1} + i\overline g_1 \overline{TZg_1} =-2iTg_1\overline{Zg_1}.
\end{eqnarray}

Since $g^\ast [q\mathrm{d}z^2]_\omega = [(q\circ g)\left(Zg_1\right)^2 \mathrm{d}z^2]_\omega$, we need to compute\\ $2(q\circ g) \left(Zg_1\right)^2 \overline Z^2\left((q\circ g)\left(Zg_1\right)^2\right)$ and $\left(\overline Z\left((q\circ g)\left(Zg_1\right)^2\right)\right)^2$. First, using equation \ref{eq3}, we have :
\[ \overline Z \left((q\circ g ) \left(Zg_1\right)^2\right) = \overline Z q(g)Zg_1|Zg_1|^2 + 4i(q\circ g)Zg_1Tg_1.\]
Using equation \ref{eq4}, it leads to :
\begin{eqnarray*}
2(q\circ g) \left(Zg_1\right)^2 \overline Z^2 \left((q\circ g ) \left(Zg_1\right)^2\right) &= &((2q\overline Z^2 q)\circ g )\left(Zg_1\right)^2 |Zg_1|^4\\
 &&+ 8i((q\overline Zq)\circ g)\left(Zg_1\right)^2Tg_1|Zg_1|^2\\
 & &-8(q^2\circ g)\left(Zg_1\right)^2\left(Tg_1\right)^2\\
 & = & ((2q\overline Z^2 q -\left(\overline Z q\right)^2)\circ g )\left(Zg_1\right)^2 |Zg_1|^4\\
 & &+ \left(\overline Z\left((q\circ g)\left(Zg_1\right)^2\right)\right)^2.
\end{eqnarray*}
Consequently,
\begin{eqnarray*}
\Dr_{2,\h} '' g^\ast[q\mathrm{d}z^2]_\omega & = & \left( \left(2q\overline Z ^2 q - \left(\overline Zq\right)^2\right)\circ q\right)\left(Zg_1\right)^3\overline{Zg_1}|Zg_1|^2 [\mathrm{d}z^3]_\omega \otimes \omega \wedge \mathrm{d}\overline z\\
& = & g^\ast\Dr_{2,\h} '' [q\mathrm{d}z^2]_\omega.
\end{eqnarray*}
\end{proof}

\begin{proposition}\label{prop2.3.4}
Let $q$ be a non-zero quadratic differential on a spherical CR manifold $M$ such that $\Dr'_2q=0$ and $\Dr''_2 q =0$. Then, every point $p\in M$ where $q(p)\neq 0$ has a neighborhood in which $q=f^\ast \mathrm{d}w^2$ for a complex valued CR function $f$.
\end{proposition}

\begin{proof}
Let $U$ be a small neighborhood of $p$ in which $q$ doesn't vanish and $q = \varphi^\ast[\alpha\mathrm{d}z^2]_\omega$ where $\varphi : U\rightarrow U' \subset \h$ is a CR diffeomorphism and $\alpha \in C^\infty (U')$. Then, $\Dr'_{2,\h} [\alpha\mathrm{d}z^2]_\omega = 0$, $\Dr''_{2,\h} [\alpha\mathrm{d}z^2]_\omega = 0$ and $\alpha$ doesn't vanish on $U'$. Since $\alpha$ doesn't vanish on $U'$, we can choose a square root of $\alpha$, $\sqrt{\alpha}$ and consider $\beta = [\sqrt{\alpha}\mathrm{d}z]_\omega$. It is then easy to see that $\Dr \beta = 0$ so that $\beta =\Dp f$, locally around $\varphi(p)$ for a CR function $f$. Thus,
\[ [\alpha \mathrm{d}z^2]_\omega = f^\ast \mathrm{d}w^2\]
and
\[ q = \varphi^\ast[\alpha\mathrm{d}z^2]_\omega = (f \circ \varphi)^\ast \mathrm{d}w^2.\]
\end{proof}

Now, we want to get from a CR function to the complex part of a CR map. For that, we shall already answer the following question : let $f: U\subset\h \rightarrow \C$ be a CR function, under what condition on $f$ can we find a real valued function $h : U\rightarrow \R$ such that $(f,h)$ is a CR map ?

It is about solving a system of partial differential equations. Assume we have such a function $h$. Then $h+i|f|^2$ is a CR function. Consequently
\[\overline Z h = -if\overline {Zf} \text{ and, since $h$ is real, } Zh = i\overline f Zf.\]
Thus, we also have
\[ Th = |Zf|^2 + i\overline f Tf - i f \overline{Tf}.\]
To sum up, we need to solve the system :
\[\begin{cases}
Zh = i\overline f Zf\\
\overline Z h = -if\overline{Zf}\\
Th = |Zf|^2 + i\overline f Tf - i f \overline{Tf}
\end{cases}.\]
Since we are only interested by a local result, we just have to see when the real $1$-form
\[\beta = i\overline f Zf\mathrm{d} z -if\overline{Zf}\mathrm{d}\overline z +\left(|Zf|^2 + i\overline f Tf - i f \overline{Tf}\right)\omega\]
is closed. After computations, we see that $\beta$ is closed if and only if 
\[ \overline Z \left(|Zf|^2\right) = -2iTf\overline {Zf}.\]
What is interesting here is that, assuming $Zf$ doesn't vanish and denoting $Zf=\sqrt{q}$, the last equation is equivalent to 
\[ \overline Z\left(|q|^2\right) + \overline q \overline Z q = 0.\]
It then gives an operator 
\[B_{2,\h} [q\mathrm{d}z^2]_\omega = \left(\overline Z\left(|q|^2\right) + \overline q \overline Z q\right)\omega\wedge\mathrm{d}\overline z \otimes \omega.\]
It is remarkable that this operator also commute to CR diffeomorphisms (it can be proved easily with the same kind of computations than the ones from the proof of Definition \ref{def223}). Thus, we can define :
\begin{definition}
Let $M$ be a spherical CR manifold $M$ with atlas $(U_i, \varphi_i)_i$. Then the following operator is well-defined (where $\Gamma (E)$ refers to the space of smooth sections of any line bundle $E$ over $M$ and $A$ is the line bundle defined in remark \ref{rk1}):
\[\begin{array}{ccccc}
B_2  : & \Gamma\left(\left(\wedge^{1,0}\right)^{\otimes 2}\right) & \longrightarrow & \Gamma\left(\wedge^{1,1} \otimes A\right) & \\
& q & \longmapsto & \varphi_i^\ast B_{2,\h}[q_i\mathrm{d}z^2]_\omega & \text{ on $U_i$ where $q=\varphi_i ^\ast[q_i\mathrm{d}z^2]_\omega$}
\end{array}.\]
\end{definition}

Now that the operators are defined, we are in position to prove Proposition \ref{th221}.
\begin{proof}[Proof of Proposition \ref{th221}]
Let $q$ be a non-zero quadratic differential on a spherical CR manifold $M$ and take a small neighborhood $U$ of a point $p\in M$ where $q$ doesn't vanish and where $q=\varphi^\ast [\alpha\mathrm{d}z^2]_\omega$ for a chart $(U,\varphi)$ and a function $\alpha\in C^\infty(\varphi(U))$. Then, $\Dr_{2,\h} ' [\alpha\mathrm{d}z^2]_\omega = 0$ and $\Dr_{2,\h} '' [\alpha\mathrm{d}z^2]_\omega = 0$ so, according to Proposition \ref{prop2.3.4}, there is locally a CR function $f$ such that $\alpha = \left(Zf\right)^2$. Now, $B_{2,\h} [\left(Zf\right)^2\mathrm{d}z^2]_\omega = 0$ ensures that the real form $i\overline f Zf\mathrm{d} z -if\overline{Zf}\mathrm{d}\overline z +\left(|Zf|^2 + i\overline f Tf - i f \overline{Tf}\right)\omega$ is closed. Consequently, there is locally a real valued function $h$ such that $(f,h)$ is a CR map. Finally, since $Zf(\varphi(p))\neq 0$, the inverse function theorem ensures that $(f,h)$ is a diffeomorphism on a small neighborhood $V$ of $\varphi(p)$. Thus,
\[ [\alpha\mathrm{d}z^2]_\omega = g^\ast [\mathrm{d}z^2]_\omega\]
where $g=(f,h)$ and so, $((g\circ \varphi)^{-1} (V), g\circ \varphi)$ is a chart around $p$ where
\[ q = (g\circ\varphi)^\ast[\mathrm{d}z^2]_\omega.\]
For the second part, one just has to notice that if $g=(g_1,g_2)$ is a CR diffeomorphisms between open subset of the Heisenberg group satisfying $\left(Zg_1\right)^2 = 1$, then $g:(z,t) \mapsto (\pm z+z_0, t+t_0\pm 2\Im(z\overline z_0))$.
\end{proof}

\subsection{Extension to forms of positive order}\label{sec2.3}

Let $k\ge 2$ be an integer. We call {\it form of degree $k$} on a strictly pseudoconvex CR manifold any smooth section of the line bundle $\left(\wedge^{1,0}\right)^{\otimes k}$. As for quadratic differentials, we can understand those in terms of cocycle relations in the spherical CR case.

\begin{definition}[$k$-form on a spherical CR manifold]
A form of degree $k$ on a spherical CR manifold $M$ with spherical CR atlas $(U_i,\varphi_i : U_i \rightarrow U_i '\subset \h)_{i\in I}$ is a collection of smooth functions $\alpha_i \in C^{\infty}(U_i ')$ satisfying for every $i,j \in I$ 
\[\alpha_i = \left(\alpha_j \circ g_{j,i}\right)\left(Zg_{j,i}^1\right)^k \text{ on $\varphi_i(U_i\cap U_j)$ where $g_{j,i} = \varphi_j \circ \varphi_i^{-1}$.}\]
\end{definition}

The idea behind operators $\Dr_2 '$ and $\Dr_2 ''$ given in remark \ref{rk3} can be used to define operators on $k$-forms. Namely, let $[\alpha\mathrm{d}z^k]_\omega$ be a $k$-form on an open subset of $\h$ and define :
\[\Dr_{k,\h} ' [\alpha \mathrm{d}z^k]_\omega = \left(k\alpha Z\overline Z \alpha + (1-k)Z\alpha \overline Z \alpha -2ik\alpha T\alpha\right)[\mathrm{d}z^{2k-1}]_\omega \otimes \omega \wedge \mathrm{d}z \text{ and}\]
\[\Dr_{k,\h} '' [\alpha \mathrm{d}z^k]_\omega = \left(k\alpha \overline Z^2 \alpha + (1-k)\left(\overline Z\alpha\right)^2\right)[\mathrm{d}z^{2k-1}]_\omega \otimes \omega \wedge \mathrm{d}\overline z.\]
Then, one can easily verify that for every CR diffeomorphism $g$ we have
\[ \Dr_{k,\h} ' g^\ast = g^\ast \Dr_{k,\h} ' \text{ and } \Dr_{k,\h} '' g^\ast = g^\ast \Dr_{k,\h} ''.\]
Moreover, the operator $B_{2,\h}$ can also be extended to $k$-forms as :
\[B_{k,\h} [\alpha\mathrm{d}z^k]_\omega = \left(\overline Z\left(|\alpha|^2\right)+\overline \alpha \overline Z \alpha\right)\omega \wedge\mathrm{d}\overline z \otimes \omega^{k-1}\]
and we still have
\[ B_{k,\h} g^\ast =g^\ast B_{k,\h}\]
for every CR diffeomorphism $g$. Thus, $\Dr_{k,\h} '$, $\Dr_{k,\h}''$ and $B_{k,\h}$ can be used to define operators on $k$-forms on a spherical CR manifold.

\begin{definition}
Let $(U_i,\varphi_i)_i$ be a spherical CR atlas of a manifold $M$ and $\alpha$ be a $k$-form on $M$. Then the following operators are well define
\[\begin{array}{ccccc}
\Dr_{k} ' : &\Gamma\left(\left(\wedge^{1,0}\right)^{\otimes k}\right) & \longrightarrow & \Gamma \left(\left(\wedge^{1,0}\right)^{\otimes 2k-1} \otimes \wedge^{2,0}\right) & \\
& \alpha & \longmapsto & \varphi_i ^\ast \Dr_{k,\h} ' [\alpha_i \mathrm{d}z^k]_\omega & \text{where $\alpha = \varphi_i ^\ast[\alpha_i\mathrm{d}z^k]_\omega$,}
\end{array}\]
\[\begin{array}{ccccc}
\Dr_{k} '': &\Gamma\left(\left(\wedge^{1,0}\right)^{\otimes k}\right) & \longrightarrow & \Gamma \left(\left(\wedge^{1,0}\right)^{\otimes 2k-1} \otimes \wedge^{1,1}\right) & \\
& \alpha & \longmapsto & \varphi_i ^\ast \Dr_{k,\h} '' [\alpha_i \mathrm{d}z^k]_\omega & \text{where $\alpha = \varphi_i ^\ast[\alpha_i\mathrm{d}z^k]_\omega$}
\end{array}\]
and
\[\begin{array}{ccccc}
B_{k} : &\Gamma\left(\left(\wedge^{1,0}\right)^{\otimes k}\right) & \longrightarrow & \Gamma \left(\wedge^{1,1}\otimes A^{k-1}\right) & \\
& \alpha & \longmapsto & \varphi_i ^\ast B_k [\alpha_i \mathrm{d}z^k]_\omega & \text{where $\alpha = \varphi_i ^\ast[\alpha_i\mathrm{d}z^k]_\omega$;}
\end{array}\]
\end{definition}

Moreover, as for quadratic differentials, the following holds.
\begin{proposition}
Let $\alpha$ be a non-zero $k$-form on a spherical CR manifold such that $\Dr_k ' \alpha = 0$ and $\Dr_k '' \alpha = 0$. Then, every point $p\in M$ where $\alpha(p)\neq 0$ has a neighborhood in which 
\[ \alpha = f^\ast \mathrm{d}w^k\]
for a complex valued CR function $f$. 

In addition, if we also have $B_k \alpha = 0$, then around every point $p$ where $\alpha$ doesn't vanish, there is a chart $(U,\varphi)$ in which 
\[\alpha = \varphi^\ast[\mathrm{d}z^k]_\omega.\]
\end{proposition}

\section{Quasiconformal mappings preserving trajectories of quadratic differentials}\label{sec3}
\subsection{Quasiconformal mappings and moduli of curve families in the Heisenberg group}
Before getting to the examples, we briefly present the theory of quasiconformal maps in the Heinseberg group, for details, refer to \cite {KR1, KR2}. The Heisenberg group is endowed with a left invariant metric 
\[ d_{\h} (p,q) := \| p^{-1} q \|_{\h}\]
where $\| (z,t) \|_{\h} := \left( |z|^4 + t^2 \right)^{\frac{1}{4}}$ is the Heisenberg norm. By analogy with the classical case, a homeomorphism $f : \Omega \longrightarrow \Omega '$ between domains of $\h$ is called quasiconformal if 
\[ H(p,f) := \underset{r\to 0} {\lim \sup} \ \frac{\underset{d_{\h} (p,q) = r } {\max}\  d_{\h} (f(p),f(q))}{\underset{d_{\h} (p,q) = r } {\min}\  d_{\h} (f(p),f(q))}, \ p\in \Omega\]
is uniformly bounded. We say that $f$ is $K$-quasiconformal if $\|H(.,f)\|_{L^{\infty}} \le K$. As in the case of the complex plane, we have equivalent analytic definitions of quasiconformality. A sufficiently regular ($C^2$ is enough) quasiconformal map between domains of $\h$ has to be a contact map for the contact structure induced by the form $\omega = \mathrm{d}t - i\overline z \mathrm{d}z + i z\mathrm{d}\overline z$, meaning that $f^* \omega = \lambda \omega$ for a nowhere vanishing real function $\lambda$. Moreover, denoting $f = (f_1 , f_2 )$ with $f_1$ the complex part of the map and $f_2$ the real one, then, if $f$ is an orientation-preserving quasiconformal map, it satisfies a system of PDEs quite similar to Beltrami equation. Indeed, in that case, there is a complex valued function $\mu \in L^{\infty}$ (called Beltrami coefficient) with $\| \mu \|_{L^\infty} < 1$ such that
\[ \overline Z f_1 = \mu Z f_1 \text{ and } \overline Z \left( f_2 + i|f_1 |^2 \right) = \mu Z \left( f_2 + i|f_1 |^2 \right) \text{ a.e.} \]
Define the distortion function of the map $f$ by
\[ K(p,f) := \frac{1 + |\mu (p)|}{1-|\mu (p)|} = \frac{ |Zf_1 (p) | + |\overline Z f_1 (p)|}{|Zf_1 (p) | - |\overline Z f_1 (p)|}\]
for $p \in \Omega$ where it makes sense and the maximal distortion of $f$ by 
\[K_f := \underset{p\in \Omega} \esssup K(p,f).\]
It is known that a conformal (i.e. $1$-quasiconformal) map $f : \Omega \longrightarrow \Omega '$ is the restriction to $\Omega$ of the action of an element of $SU(2,1)$ (see \cite [p.~337]{KR1} for the smooth case and \cite[p.~869] {Cap} for the general one).
\newline

By analogy with the case of the complex plane, in order to understand extremal properties of a quasiconformal map between domains of $\h$, we look at its behaviour on a well chosen family of Legendrian curves that foliates the domain. First, let $\gamma = (\gamma_1 , \gamma_2) : I \longrightarrow \h$ be a $C^1$ curve. $\gamma$ is Legendrian if and only if
\[ \dot \gamma_2 (s) = -2 \Im (\overline \gamma_1 (s) \dot \gamma_1 (s)) \text{ for every $s\in I$.}\]
Then, let $\Gamma$ be a family of Legendrian curves in a domain $\Omega$ of $\h$. Denote $\adm (\Gamma)$ the set of measurable Borel functions $\rho : \Omega \longrightarrow [0,\infty]$ such that
\[ \int_\gamma \rho \mathrm{d} l = \int_a ^b \rho(\gamma (s)) |\dot \gamma_1 (s)| \mathrm{d} s \ge 1\]
for every curve $\gamma : ]a,b[ \longrightarrow \Omega$ in $\Gamma$. Elements of $\adm (\Gamma)$ are called {\it densities}. Define the {\it modulus} of $\Gamma$ by
\[ M(\Gamma) = \underset{\rho\in \adm (\Gamma)}\inf \int_\Omega \rho^4 \mathrm{d} L^3\]
where $\mathrm{d} L^3$ is the Lebesgue measure on $\R ^3$. A density $\rho_0$ is said {\it extremal for $\Gamma$} if
\[ M(\Gamma) = \int_\Omega \rho_0 ^4 \mathrm{d} L^3.\]
Then, we have the following theorem (see e.g. \cite [p.~177]{BFP}).

\begin{theorem}\label{th3.1.1}
Let $f : \Omega \longrightarrow \Omega '$ be a quasiconformal map between domains in $\h$ and $\Gamma$ a family of Legendrian curves in $\Omega$. Then
\begin{enumerate}
\item{for all $\widetilde \rho \in \adm (f(\Gamma))$,
\[M(\Gamma) \le \int_{\Omega '} K(f^{-1}(x), f)^2\widetilde \rho ^4 (x) \mathrm{d} L^3 (x),\]}
\item{for all $\rho \in \adm (\Gamma)$,
\[ M(f(\Gamma)) \le \int_\Omega K(p, f)^2 \rho^4 (p) \mathrm{d} L^3 (p),\]}
\item{ and so
\[ \frac{1}{K_f ^2} M(\Gamma) \le M(f(\Gamma)) \le K_f ^2 M(\Gamma).\]}
\end{enumerate}
\end{theorem}

Let $\Omega, \Omega'$ be domains in $\h$ and $\mathcal F$ be a class of quasiconformal maps from $\Omega$ to $\Omega '$. Let $f_0$ be an element of $\mathcal F$. We say that $f_0$ minimizes the maximal distortion in $\mathcal F$ if
\[ K_{f_0} ^2 = \underset{f\in \mathcal F} \min \ K_f ^2\]
and $f_0$ minimizes the mean distortion in $\mathcal F$ for the density $\rho_0$ if
\[ \int_\Omega K(p,f_0)^2 \rho_0 ^4 (p) \mathrm{d} L^3 (p) = \underset{f \in \mathcal F} \min \int_ \Omega K(p,f)^2 \rho_0 ^4 (p) \mathrm{d} L^3 (p).\]

Using quasiconformal mappings in the Heisenberg group, one can define the Teichmüller space of a spherical CR manifold (see \cite{Wan}) so that understanding quasiconformal mappings minimizing the maximal distortion in a class of quasiconformal mappings (for instance an isotopy class) is of general interest. In that direction, Tang, in \cite{Tan}, constructed and prove uniqueness of a quasiconformal map with minimal distortion between CR circle bundles over flat tori. In \cite{BFP, BFP2, Tim}, the authors constructed and prove uniqueness of quasiconformal mappings minimizing a mean distortion functional between several domains of the Heisenberg group. As in the case of Riemann surfaces, these results were all obtained using well chosen family of curves. 

\subsection{Quasiconformal maps dilating trajectories of a quadratic differential}

As an end, we provide examples of quasiconformal maps between domains of $\h$ minimizing the maximal distortion or a mean distortion functional in a class of quasiconformal maps and explicit the quadratic differentials involved in those examples. We won't deal with links between extremal densities and quadratic differential annihilated by the operators defined in section \ref{sec2.2}, it would be a subject for other investigations in the future. We just notice that all quadratic differentials involved in the following examples annihilate both $\Dr_2 '$ and $\Dr_2 ''$ which comes from their definitions as pull-backs by CR functions of the quadratic differential $\mathrm{d}w^2$ on $\C$.

\begin{example}
First, we begin with an example of a quasiconformal map between domains of $\h$ with constant distortion. This example was given in \cite [p.~171]{BFP} and is constructed using the projection $P(z,t) = z$. Consider a rectangle $R_{a,b} = \{ z \in \C \ | \ 0< \Re (z) < a, \ 0 < \Im (z) < b\}$ with $a,b > 0$ foliated by vertical lines
\[\delta_x (s) = x+is\]
for $x \in ]0,a[$ and $s \in ]0,b[$. Lifting such a vertical line to a Legendrian curve in $\h$ gives a $1$-parameter family of curves 
\[\widetilde \delta _{x, t} (s) = (x+is, t-\Im ((x+is)^2)).\] In this way, we obtain a domain
\[ \Omega = \{ (z, t-\Im (z^2)) \in \h \ | \ z \in R_{a,b} , \ 0<t<c \}\]
where $c >0$. And the curves $\widetilde \delta _{x, t}$ are vertical trajectories for the quadratic differential 
\[q = [\mathrm{d} z^2]_\omega.\]
Horizontal trajectories for $q$ lying in $\Omega$ are the curves 
\[\widetilde \delta_{y,t} (s) = (s+iy, t+2sy)\]
for $s \in ]0,a[$, $y\in ]0,b[$ and $t+4ys \in ]0, c[$.
Consider another domain
\[ \Omega ' = \{ (z, t-\Im (z^2)) \in \h \ | \ z \in R_{a',b'} , \ 0<t<c' \}\]
with $a',b',c' >0$ such that
\[\frac{c'}{c} = \frac{a'b'}{ab}.\]
For $y \ge 0$, denote
\[ \partial \Omega_y = \{ (z,t-\Im (z^2)) \ | \ \Im (z) = y\}\]
and consider the class $\mathcal F$ of all quasiconformal mappings $f : \Omega \longrightarrow \Omega '$ which extend homeomorphically to the boundary with
\[ f(\partial\Omega_0) = \partial \Omega '_0 \text{ and } f (\partial \Omega_b) = \partial \Omega '_{b'}.\]
Then, the map $f_0$ defined by
\[ f_0 (x+iy, t) = \left( \frac{a'}{a} x + i \frac{b'}{b} y, \frac{a'b'}{ab} t\right)\]
is a quasiconformal map in $\mathcal F$ which minimizes the maximal distortion in $\mathcal F$. Moreover, it is clear that $f_0$ dilates the vertical and horizontal trajectories of the quadratic differential $[\mathrm{d}z^2]_\omega$ which is what is expected of a Teichm\"uller homeomorphism.
\end{example}

\begin{example}
The second example is the counterpart of the previous example using the projection $\Pi (z,t) = t+i|z|^2$ (which is a CR function). Consider again a rectangle $R_{a,b} = \{ z \in \C \ | \ 0< \Re (z) < a, \ 0 < \Im (z) < b\}$ with $a,b > 0$ foliated by vertical lines 
\[\delta_x (s) = x+is\]
for $x \in ]0,a[$ and $s \in ]0,b[$. Lifting such a vertical line to a Legendrian curve in $\h$ gives the $1$-parameter family of cylindrical radii:
\[ \widetilde \delta_{z, t} (s) = (sz,t)\]
with $|z|=1$, $0<t<a$ and $0<s<\sqrt b$. For $0<c<2\pi$, consider then the domain
\[C = \{ (z,t) \in \h \ | \ t+i|z|^2 \in R_{a,b}, \ 0<\arg (z)<c\}.\]
The curves $\widetilde \delta_{z,t}$ for $|z|=1$, $0<\arg(z)<c$ and $t \in ]0,a[$ are the vertical trajectories of the quadratic differential 
\[q = \Pi^\ast\mathrm{d}w^2 = [-4\overline z ^2 \mathrm{d} z^2]_\omega\]
on $C$. Horizontal trajectories for $q$ lying in $C$ are
\[ \widetilde \gamma_z (s) = \left( ze^{-i \frac{s}{2|z|^2}}, s \right)\]
for $0<|z|<\sqrt b$, $0<s<a$ and $0< \arg (z) -\frac{s}{2|z|^2} < c$. Consider another domain
\[ C' =\{ (z,t) \in \h \ | \ t+i|z|^2 \in R_{a',b'}, \ 0<\arg (z) < c' \}\]
with $a',b' > 0$, $0<c'<2\pi$ such that
\[ \frac{bc}{a} = \frac{b'c'}{a'}.\]
Denote
\[ \partial C_y = \{ (z,t) \in \overline{C} \ | \ |z|= \sqrt{y} \}\]
and consider the class $\mathcal F$ of all quasiconformal maps $f : C \longrightarrow C'$ which extend homeomorphically to the boundary with
\[ f(\partial C_0) = \partial C'_0 \text{ and } f(\partial C_b) = \partial C'_{b'}.\]
Then, the map $f_0$ defined by
\[ f_0 (z,t) = \left( \sqrt{\frac{b'}{b}} |z|e^{i \frac{a'b}{ab'} \arg (z)}, \frac{a'}{a} t\right)\]
is a quasiconformal map in $\mathcal F$ which minimizes the maximal distortion in $\mathcal F$.
\newline
Indeed, let $\Delta_0$ be the family of curves $\widetilde \delta_{z,t} : ]0,\sqrt b[ \longrightarrow C$ with $|z|=1$, $0<t<a$. Then, one can compute
\[ M(\Delta_0) = \frac{8 ac}{27 b}\]
with extremal density
\[ \rho_0 (z,t) = \frac{2}{3b^{\frac{1}{3}} |z|^{\frac{1}{3}}}.\]
Moreover, let $\Delta$ be the family of all Legendrian curves connecting $\partial C_0$ and $\partial C_b$. Then, since $\rho_0 \in \adm (\Delta)$, it is extremal for $\Delta$. Thus,
\[ M(\Delta) = \frac{8 ac}{27 b}.\]
Denote $\Delta '$ the corresponding family in $C'$. Then, by hypothesis on $\mathcal F$, for every $f \in \mathcal F$ one has
\[ f(\Delta) = \Delta '\]
and the third point in Theorem \ref{th3.1.1} gives for every $f \in \mathcal F$
\[ \frac{1}{K_f ^2} \frac{8 ac}{27 b} \le \frac{8 a' c' }{27 b'} \le K_f ^2 \frac{8 ac}{27 b}.\]
Thus, since $\frac{bc}{a} = \frac{b'c'}{a'}$, one has, for every $f \in \mathcal F$, 
\[ K_f ^2 \ge \max \left( \left(\frac{a'b}{ab'}\right)^2 , \left( \frac{ab'}{a'b} \right)^2 \right) = K_{f_0} ^2.\]
Meaning precisely that $f_0$ minimizes the maximal distortion in $\mathcal F$.
\newline
Moreover, we easily see that $f_0$ dilates the vertical trajectories of $[-4\overline z ^2 \mathrm{d} z^2]_\omega$ by a factor $\sqrt{\frac{b'}{b}}$ and the horizontal trajectories by a factor $\frac{a'}{a}$.
\end{example}

The next examples concern quasiconformal maps minimizing a mean distortion functional. They are the main objects considered in \cite {BFP, BFP2, Tim}. Most of what follows is a redraft of \cite{Tim} in terms of trajectories quadratic differentials.
\newline

We remind that a  {\it quadrilateral} $(Q,I_1,I_2)$ is the data of a subset $Q$ of $\C$ homeomorphic to a closed disk, with two distinguished parts of its boundary, $I_1$ and $I_2$ that are disjoint, connected, non-empty and not reduced to a single point. Such objects can be brought in the normal form of a rectangle $\overline {R_{a,b}}$ as in the previous example. That is, there is a homeomorphism $\phi : Q \rightarrow \overline{R_{a,b}}$ which is holomorphic in the interior of $Q$ and sends $I_1$ onto $\{0\}\times [0,b]$ and $I_2$ onto $\{a\} \times [0,b]$. Moreover, the rectangle and $\phi$ are unique up to dilation. A quasiconformal map between the quadrilaterals $(Q,I_1, I_2)$ and $(Q' , I_1 ', I_2 ')$ is a homeomorphism from $Q$ onto $Q'$ which is quasiconformal in the interior of $Q$ and sends $I_i$ onto $I_i '$ for $i=1,2$.

In what follows, we will consider two quadrilaterals $(Q, I_1, I_2)$ and $(Q', I_1 ', I_2 ')$ lying inside $\overline {\mathbb H}$ where $\mathbb H$ is the upper half-plane. On these quadrilaterals, there are natural holomorphic quadratic differentials to consider that come from uniformization by rectangles (unique up to dilation which doesn't change the mains features such as horizontal/vertical trajectories). With $\phi : Q\rightarrow \overline{R_{a,b}}$ and $\psi : Q' \rightarrow \overline{R_{a', b'}}$ denoting the uniformizations, these holomorphic quadratic differentials are :
\[q_Q = \phi ^\ast \mathrm{d}w^2 \text{ and } q_{Q'} = \psi ^\ast \mathrm{d}w^2.\]

On $Q$ (resp. $Q'$) we consider $\Gamma_Q$ (resp. $\Gamma_{Q'}$) the family of horizontal trajectories of $q_Q$ (resp. of $q_{Q'}$) lying in $Q$ (resp. in $Q'$) and connecting $I_1$ and $I_2$ (resp. connecting $I_1 '$ and $I_2 '$ ). 

\begin{lemma}
Every curve in $\Gamma_Q$ has $q_Q$-length $a$. Moreover, $\rho_Q = \frac{\sqrt{|q_Q|}}{a}$ is extremal for $\Gamma_Q$.
\end{lemma}

\begin{proof}
Let $\gamma$ be a curve in $\Gamma_Q$. Then, we can parametrize $\gamma$ by $\gamma : s \mapsto \phi^{-1} (s+iy)$ for some $y\in [0,b]$ with $s\in [0,a]$. It is then clear that $q_Q (\gamma ')= \left(\phi '(\gamma (s))\gamma'(s)\right)^2 = 1$ so that
\[l_{q_Q} (\gamma) = \int_0 ^a \mathrm{d}s = a.\]
The extremality of $\rho_Q$ follows immediately from the holomorphy of $q_Q$.
\end{proof}

Consider then the {\it lifted quadrilaterals} by the map $\Pi : (z,t) \mapsto t+i|z|^2$ : $(\Omega, J_1, J_2)$ and $(\Omega ' , J_1 ' , J_2 ')$ where $\Omega = \Pi^{-1} (Q)$, $J_i = \Pi^{-1} (I_i)$, $\Omega ' = \Pi^{-1} (Q')$ and $J_i ' = \Pi^{-1} (I_i ')$. On these lifted quadrilaterals, we define the quadratic differentials 
\[q_\Omega = \Pi^\ast q_Q = [-4\overline z^2 q_Q (t+i|z|^2)]_\omega \text{ and }q_{\Omega '} = \Pi^{\ast}q_{Q'} = [-4\overline z^2 q_{Q'} (t+i|z|^2)]_\omega.\]
As for quadrilaterals, we call a quasiconformal map between the lifted quadrilaterals $(\Omega, J_1, J_2)$ and $(\Omega ', J_1 ', J_2 ')$ a homeomorphism from $\Omega$ onto $\Omega '$ which is quasiconformal in the interior of $\Omega$ and maps $J_i$ onto $J_i'$ for $i=1,2$.

Denote $\Gamma_\Omega$ (resp. $\Gamma_{\Omega '}$) the family of horizontal trajectories of $q_\Omega$ (resp. of $q_{\Omega '}$) that connect $J_1$ and $J_2$ (resp. $J_1 '$ and $J_2 '$). 

\begin{lemma}
Every curve in $\Gamma_\Omega$ has $q_\Omega$-length $a$.
\end{lemma}

\begin{proof}
Let $\gamma = (\gamma_1 , \gamma_2)$ be a curve in $\Gamma_\Omega$. Then, there is a curve $\delta$ in $\Gamma_Q$ such that $\Pi \circ \gamma = \delta$. Moreover, we have:
\begin{eqnarray*}
q_\Omega (\gamma) & = & -4\overline \gamma_1 ^2 q_Q (\delta) (\dot\gamma_1)^2\\
& = & \left(2i \overline \gamma_1\dot \gamma_1\right)^2 q_Q (\delta)\\
& = & q_Q (\delta) (\dot \delta)^2\\
& = & 1.
\end{eqnarray*}
Consequently,
\[l_{q_\Omega} (\gamma) = l_{q_Q} (\delta) = a.\]
\end{proof}
Contrary to the complex plane case, $\frac{\sqrt{|q_\Omega|}}{a}$ has no reason to be extremal for $\Gamma_\Omega$. 
We say that $\Gamma_\Omega$ satisfies the {\it pull-back density condition} if $\frac{\sqrt{|q_\Omega|}}{a}$ is an extremal density for $\Gamma_\Omega$.

In this context, we have a general result concerning quasiconformal maps which dilate horizontal trajectories of quadratic differentials (see Proposition 3.17 in \cite{Tim} for a proof).

\begin{theorem}\label{th3.2.3}
Assume that $\Gamma_\Omega$ and $\Gamma_{\Omega '}$ satisfy the pull-back density condition. Let $f : (\Omega, J_1, J_2) \longrightarrow (\Omega ', J_1 ', J_2 ')$ be a $C^2$ quasiconformal map that dilates horizontal trajectories of $(q_\Omega, q_{\Omega '})$. Then, there is a $C^2$ quasiconformal map $g : Q \longrightarrow Q '$ such that
\[ \Pi \circ f = g \circ \Pi.\]
\end{theorem}

\begin{remark}
The map $g$ of the above theorem dilates horizontal trajectories of $(q_Q , q_{Q'})$. Moreover, according to Theorem 3.9 in \cite {Tim}, $g$ is symplectic with respect to the hyperbolic area forms on $\mathbb H$.
\end{remark}

The next examples are applications of this theorem in order to find all quasiconformal maps in appropriate classes of maps between domains of $\h$ that dilate horizontal trajectories. For cylinders, there is, up to composition with a vertical rotation, a unique quasiconformal map (in an appropriate class of quasiconformal maps) which dilates horizontal trajectories. For cylinders from which a smaller cylinder has been removed, there can be no quasiconformal map (in an appropriate class of quasiconformal maps) which dilates horizontal trajectories. Finally, for spherical annuli, there is a $2$-parameters family of quasiconformal maps (in an appropriate class of quasiconformal maps) dilating horizontal trajectories.

\begin{example}
Consider cylinders $C_{a,b}$ and $C_{a',b'}$ where
\[C_{r,R}= \{ (z,t) \in \h \ | \ |z|^2 < R, t\in ]0,r[\},\]
for $a,b,a',b' > 0$ satisfying
\[ \frac{ab'}{a'b} >1\]
and the trajectories of the quadratic differential 
\[q=\Pi^{\ast}\mathrm{d}w^2 = [-4\overline z^2 \mathrm{d} z^2]_\omega\]
where $\Pi : (z,t) \mapsto t+i|z|^2$. Its vertical trajectories are the cylindrical radii
\[ \delta_{z,t} (s) = (sz,t)\]
for $|z|=1$, $t\in ]0,a[$ and $s\in ]0,\sqrt b[$ and horizontal trajectories are the curves
\[ \gamma_z (s) = \left( z e^{-i\frac{s}{2|z|^2}}, s\right)\] 
for $0<|z|<\sqrt b$ and $s \in ]0,a[$ (see Figure 1).

\begin{figure}[!h]
\center
\includegraphics[width=10cm,height=7cm]{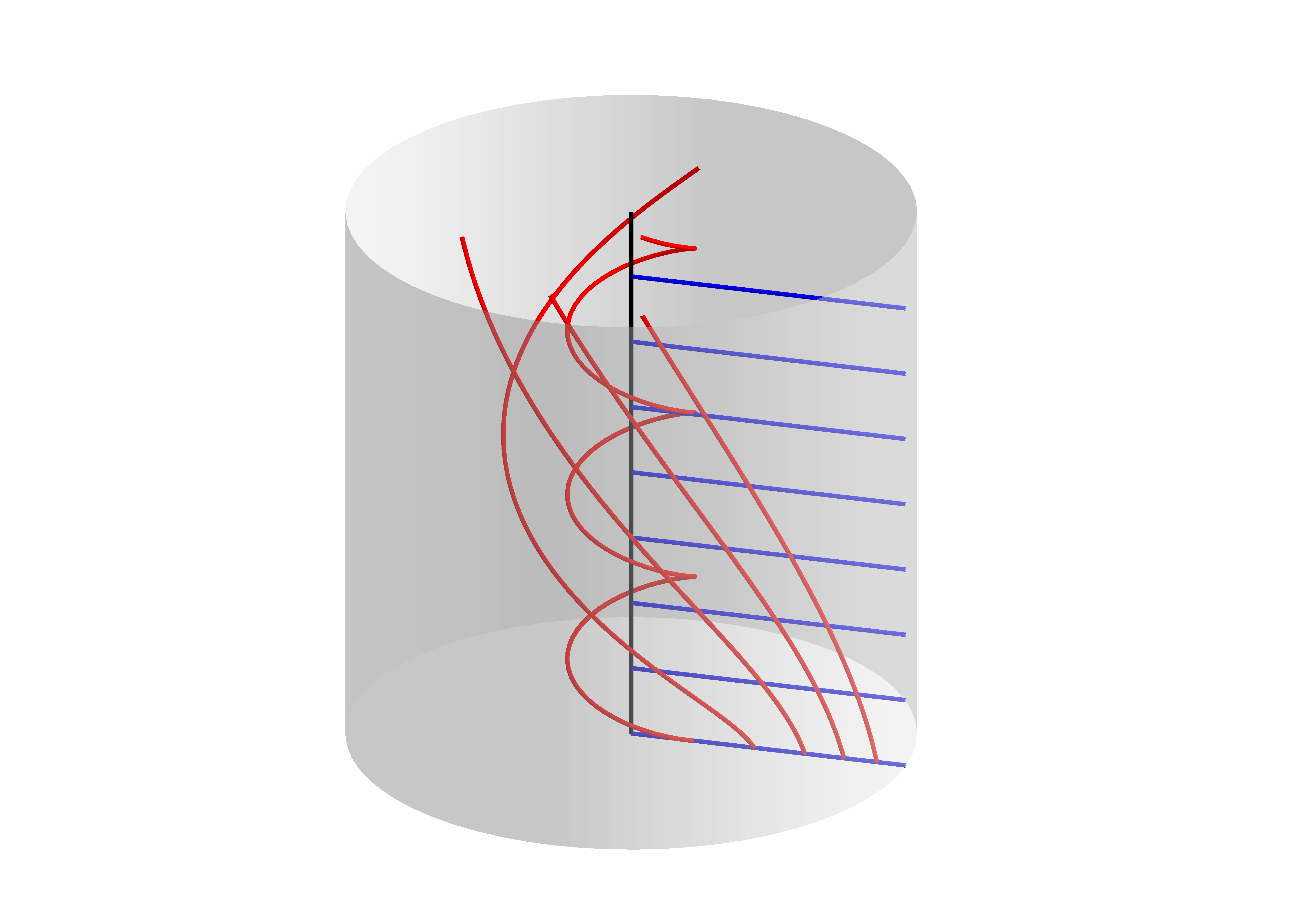}
\caption{Horizontal and vertical trajectories of the quadratic differential $q=[-4\overline z^2 \mathrm{d} z^2]_\omega$ lying in a cylinder. Vertical trajectories are rotations around the vertical axis of the cylindrical radii and horizontal are rotations around the vertical axis of the drawn curves connecting the boundary discs. }
\end{figure}

Consider the class $\mathcal F$ of quasiconformal maps from $C_{a,b}$ to $C_{a',b'}$ which extend homeomorphically to the boundary and map
\begin{itemize}
\item{the disc $\{(z,t) \in \h \ | \ |z|\le \sqrt b  \text{ \& } t = 0\}$ to the disc $\{(z,t) \in \h \ | \ |z|\le \sqrt {b'}  \text{ \& } t = 0\}$,}
\item{ and the disc $\{(z,t) \in \h \ | \ |z|\le \sqrt b  \text{ \& } t = a\}$ to the disc $\{(z,t) \in \h \ | \ |z|\le \sqrt {b'}  \text{ \& } t = a'\}$,}
\end{itemize}
Let $\Gamma_0$ be the family of all horizontal trajectories of $q$. Then, Theorem 2.9 in \cite {Tim} states that the map $f_0$ defined by 
\[ f_0 (z,t) = \left( \frac{\sqrt{b'} z e^{\frac{it}{2b} \left( 1- \frac{a'b}{ab'} \right)}}{\sqrt{\left( 1-\frac{ab'}{a'b}\right)|z|^2 + \frac{ab'}{a'}}}, \frac{a'}{a} t \right)\]
is a quasiconformal map in $\mathcal F$ which minimizes the mean distortion in $\mathcal F$ for the extremal density of $\Gamma_0$ and it is the unique such minimizer up to composition with a vertical rotation. Moreover, $f_0$ $\frac{a'}{a}$-dilates horizontal trajectory of $q$: $f$ sends the horizontal trajectory $s \longmapsto \gamma_z (s)$ to the horizontal trajectory $s \longmapsto \gamma_{z'} \left(\frac{a'}{a} s\right)$ with $z' = \frac{\sqrt{b'} z}{\sqrt{ \left( 1-\frac{ab'}{a'b} \right) |z|^2 + \frac{ab'}{a'}}}$. But it only preserves vertical trajectories of $q$: $f$ sends the vertical trajectory $s \longmapsto \delta_{z,t} (s)$ to the vertical trajectory $s \longmapsto \delta_{z' , t'} \left( \frac{\sqrt{b'} s}{\sqrt{\left( 1-\frac{ab'}{a'b} \right) s^2 + \frac{ab'}{a'}}} \right)$ with $t' = \frac{a'}{a} t$ and $z' = ze^{\frac{it}{2b} \left( 1-\frac{a'b}{ab'} \right)}$.
\newline

Theorem \ref{th3.2.3} says more: up to composition with a vertical rotation, $f_0$ is the only $C^2$ quasiconformal map in $\mathcal F$ which $\frac{a'}{a}$-dilates horizontal trajectories of $q$.
\newline
Indeed, let $f$ be such a map. Then, according to Theorem \ref{th3.2.3}, there is a quasiconformal map $g : R_{a,b} \longrightarrow R_{a',b'}$ such that $\Pi \circ f = g \circ \Pi$. Since $f$ $\frac{a'}{a}$-dilates horizontal trajectories of $q$, for every $x+iy \in R_{a,b}$,
\[ \Re (g (x+iy)) = \frac{a'}{a} x.\]
Then, Theorem 3.9 in \cite {Tim} states that $g$ must be a symplectomorphism with respect to the hyperbolic area form of the upper half-plane. Thus, $\varphi = \Im (g)$ satisfies for every $x+iy \in R_{a,b}$,
\[ \frac{a'}{a} \frac{\partial \varphi}{\partial y} (x+iy) \frac{1}{\varphi (x+iy) ^2} = \frac{1}{y^2}.\]
Solving, we find for every $x+iy \in R_{a,b}$,
\[\varphi (x+iy) = \frac{a'y}{a+ a'yh(x)}\]
for a function $h$. Moreover, since $f$ sends $\{ (z,t) \in \h \ | \ |z|= \sqrt{b} \text{ \& } 0\le t \le a \}$ to $\{ (z,t) \in \h \ | \ |z|= \sqrt{b} \text{ \& } 0\le t \le a \}$, for every $x \in [0,a]$,
\[ \varphi (x+ib) = b'\]
and so, $h$ is constant with value $b' - \frac{a}{a'b}$. Thus, $g$ is completely determined and so $f$ is also determined up to composition with a vertical rotation (this is a consequence of Theorem 3.9 in \cite {Tim}). Then, $f$ is, up to composition with a vertical rotation, $f_0$.
\newline

On the other hand, there is no $C^2$ quasiconformal map in $\mathcal F$ that $\sqrt{\frac{b'}{b}}$-dilates vertical trajectories of $q$. Indeed, let $f$ be such a map and write it in cylindrical coordinates $z=re^{i\theta}$, $t=t$ as $(R, \Theta , T)$. Then, the fact that $f$ sends vertical trajectories to $s \longmapsto \delta_{z,t} \left( \sqrt{\frac{b'}{b}} s \right)$ implies that
\[ \frac{\partial T}{\partial r} = \frac{\partial \Theta}{\partial r} = 0 \text{ and } R = \sqrt{\frac{b'}{b}} r.\]
The contact form $\omega$ in cylindrical coordinates is
\[ \mathrm{d} t + 2r^2 \mathrm{d} \theta.\]
Since $f$ is contact, $R,\Theta$ and $T$ satisfy
\begin{eqnarray}
\frac{\partial T}{\partial t} + 2 \frac{b'}{b} r^2 \frac{\partial \Theta}{\partial t} & = & \frac{1}{2r^2} \frac{\partial T}{\partial \theta} + \frac{b'}{b} \frac{\partial \Theta}{\partial \theta}. 
\end{eqnarray}
Then, differentiating it with respect to $r$ gives
\[ 2\frac{b'}{b} r^2 \frac{\partial \Theta}{\partial t} = -\frac{1}{2r^2} \frac{\partial T}{\partial \theta}\]
and reporting this expression in $(2)$, we obtain
\[ \frac{\partial T}{\partial t} + \frac{1}{r^2} \frac{\partial T}{\partial \theta} = \frac{b'}{b} \frac{\partial \Theta}{\partial \theta}\]
leading to
\[ \frac{\partial T}{\partial \theta} = 0.\]
This, together with the fact that $\frac{\partial R}{\partial \theta} = 0$ implies that there is a map $g : R_{a,b} \longrightarrow R_{a',b'}$ such that
\[ \Pi \circ f = g \circ \Pi.\]
Since $R = \sqrt{\frac{b'}{b}} r$, for every $x+iy \in R_{a,b}$
\[ \Im (g(x+iy)) = \frac{b'}{b} y.\]
Again, Theorem 3.9 in \cite {Tim} implies that $g$ is a symplectomorphism with respect to the hyperbolic area form of the upper half-plane which leads to
\[ \frac{\partial \Re (g)}{\partial x} = \frac{b'}{b}.\]
Since $g$ maps $\{ w \in \overline {R_{a, b}} \ | \ \Re (w) = 0 \}$ on $\{ w \in \overline {R_{a',b'}} \ | \ \Re (w) = 0\}$, it must be the dilation with factor $\frac{b'}{b}$ which is impossible by the hypothesis $\frac{ab'}{a'b} >1$. 
\newline

Consider two domains
\[D_{a, b} = \{ (z,t) \in \h \ | \ 0<t<a, \ 1< |z|^2 < b+1\} \text{ and}\]
\[D_{a', b'} = \{ (z,t) \in \h \ | \ 0<t<a', \ 1< |z|^2 < b'+1\}\]
for $a, b, a', b' > 0$ with
\[\frac{a(b'+1)}{a'(b+1)} > 1\]
and the trajectories of the quadratic differential 
\[q=[-4\overline z^2 \mathrm{d} z^2]_\omega.\]
Let $\mathcal F$ be the class of quasiconformal maps from $D_{a,b}$ to $D_{a',b'}$ which extend homeomorphically and map the two boundary annuli to their corresponding ones and the two boundary cylinders to their corresponding ones. Denote $\gamma_z$ the horizontal trajectories of $q$. Then, Example 3.20 in \cite {Tim} states that a quasiconformal map $g_0 \in \mathcal F$ which $\frac{a'}{a}$-dilates horizontal trajectories of $q$ exists if and only if 
\[\frac{ab}{b+1} = \frac{a'b'}{b'+1}\]
(meaning that the rectangles $\{ w \in \C \ | \ 0< \Re (w) < a, \ 1 < \Im (w) < b+1\}$ and $\{ w \in \C \ | \ 0< \Re (w) < a', \ 1 < \Im (w) < b'+1\}$ have the same hyperbolic area). If it is the case, then $g_0$ is, up to composition with a vertical rotation, the restriction to $D_{a,b}$ of the map $f_0 : C_{a, b+1} \longrightarrow C_{a', b'+1}$ given previously (with $b$ replaced by $b+1$ and $b'$ by $b'+1$).
\end{example}

\begin{example}
For the last example, consider two spherical annuli $A_a$ and $A_{a^k}$ for $a > 1$ and $0<k<1$, where
\[ A_r = \{ (z,t) \in \h \ | \ 1 < \| (z,t) \|_\h < r \},\]
foliated by radial curves
\[ \gamma_{y,\alpha} (s) = \left( \sqrt{e^s \sin (y)} \alpha e^{-\frac{is}{2} \cot (y)}, e^s \cos (y)\right)\]
for $|\alpha|=1$, $y \in ]0,\pi[$. These curves are horizontal trajectories for the quadratic differential 
\[q = \Pi^\ast \left(\frac{\mathrm{d}w^2}{w^2}\right) = [\frac{-4\overline z ^2}{(t+i|z|^2)^2} \mathrm{d} z^2]_\omega\]
where $\Pi : (z,t) \mapsto t+i|z|^2$; and its vertical trajectories are spherical arcs
\[ \delta_{x,\alpha} (s) = \left(  \sqrt{e^x \sin (s)} \alpha e^{\frac{is}{2}} , e^x \cos (s) \right)\]
for $|\alpha| = 1$, $x \in ]0, 2 \log a[$ (see Figure 2). 

\begin{figure}\label{fig2}
\center
\includegraphics[width=15cm,height=7cm]{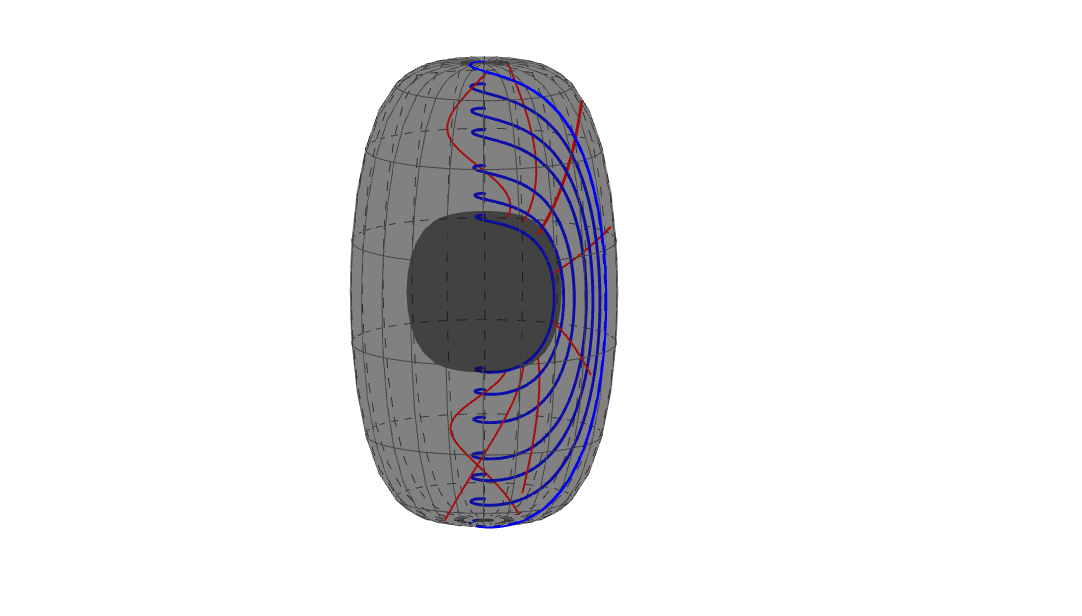}
\caption{Horizontal and vertical trajectories for the quadratic differential $q = \left[\frac{-4\overline z ^2}{(t+i|z|^2)^2} \mathrm{d} z^2\right]_\omega$ lying in an annulus. Horizontal trajectories are rotations around the vertical axis of the curves connecting the boundary spheres and vertical trajectories are rotations around the vertical axis of the spherical arcs. }
\end{figure}

It will be more convenient here to write in {\it logarithmic coordinates} introduced in \cite {Pla}:
\[ (z,t) = \left(i\cos ^{\frac{1}{2}} \psi e^{\frac{\xi + i (\psi - 3\eta)}{2}} , - \sin \psi e^\xi \right)\]
for $\xi \in \R$, $\psi \in [-\frac{\pi}{2} , \frac{\pi}{2} ]$ and $\psi - 3 \pi \le 2\eta < \psi + \pi$. In those coordinates, up to parametrization, horizontal trajectories of $q$ are
\[ \widetilde\gamma_{\psi , \eta} (s) = \left( s , \psi , \eta - \frac{\tan \psi}{3} s \right)\]
and its vertical trajectories are
\[  \widetilde \delta_{\xi , \eta} (s) = \left( \xi , s , \eta\right).\]
Let $\mathcal F$ be the class of all quasiconformal maps from $A_a$ to $A_{a^k}$ which extend homeomorphically to the boundary and map the boundary Kor\'anyi spheres of $A_a$ to the respective boundary Kor\'anyi spheres of $A_{a^k}$. Let $f_k$ be the radial stretch map defined in logarithmic coordinates by
\[ f_k (\xi , \psi , \eta) = \left(k\xi , \tan ^{-1} \left(\frac{\tan \psi}{k}\right), \eta \right).\]
Then, Theorem 2 in \cite {BFP} states that $f_k$ minimizes the mean distortion in $\mathcal F$ for the extremal density of the family of horizontal trajectories of $q$ in $A_a$. In \cite {BFP2} the authors proved that, up to composition with a vertical rotation, $f_k$ the only such minimizer. 
Moreover, $f_k$ $k$-dilates horizontal trajectories of $q$: $f_k$ maps the horizontal trajectory $s \longmapsto \widetilde \gamma_{\psi , \eta} (s)$ to the horizontal trajectory $s \longmapsto \widetilde \gamma_{\tan ^{-1} \left( \frac{\tan \psi}{k} \right) , \eta} (ks)$; and preserves vertical trajectories of $q$: $f_k$ sends the vertical trajectory $s \longmapsto \widetilde \delta_{\xi , \eta} (s)$ to the vertical trajectory $s \longmapsto \widetilde \delta_{k\xi , \eta} \left( \tan ^{-1} \left(\frac{ \tan s}{k} \right) \right)$.
\newline

However, contrary to the case of cylinders, composition with vertical rotations of $f_k$ are not the only quasiconformal maps in $\mathcal F$ which $k$-dilates horizontal trajectories of $q$. Indeed, for every real number $D$, the map defined in logarithmic coordinates by
\[ g_{D}: \left( \xi , \psi , \eta \right) \longmapsto \left( k\xi , \tan ^{-1} \left( \frac{\tan \psi}{k} + D\right), \eta - \frac{kD\xi}{3} \right)\]
is a quasiconformal map in $\mathcal F$ which sends  $s \longmapsto \widetilde \gamma_{\psi , \eta} (s)$ to $s \longmapsto \widetilde \gamma_{\psi ' , \eta '} (ks)$. Up to composition with a vertical rotation, these maps $g_D$ are the only $C^2$ quasiconformal maps in $\mathcal F$ which $k$- dilate horizontal trajectories of $q$ (this is a consequence of Theorem \ref{th3.2.3}).
\newline

On the other hand, as in the case of cylinders, there is no $C^2$ quasiconformal map in $\mathcal F$ which $1$-dilate vertical trajectories of $q$. The proof of this goes along the same lines as the one for cylinder, except that it is more convenient to use logarithmic coordinates. 
\end{example}

\Addresses

\end{document}